\documentclass[12pt]{amsart}
\usepackage{amssymb}
\usepackage{verbatim}
\usepackage[usenames]{color}
\usepackage{hyperref}
\usepackage{url}
\usepackage{mathrsfs}

\newtheorem{thm}{Theorem}
\newtheorem{prop}[thm]{Proposition} 
\newtheorem{la}[thm]{Lemma}
\newtheorem{cor}[thm]{Corollary}

\theoremstyle{definition}
\newtheorem{df}[thm]{Definition}

\newtheorem{notat}[thm]{Notation}
\newtheorem{prob}[thm]{Problem}

\theoremstyle{remark} 
\newtheorem{ex}[thm]{Example} 
\newtheorem{rmk}[thm]{Remark}

\newenvironment{ls}{\begin{itemize}}{\end{itemize}}
\newenvironment{lsnum}{\begin{enumerate}}{\end{enumerate}}
\newenvironment{pf}{\begin{proof}}{\end{proof}}
\newcommand{\ger}[1]{\ensuremath{\mathfrak {#1}}}
\newcommand{\scr}[1]{\ensuremath{\mathcal {#1}}}
\newcommand{\bld}[1]{\ensuremath{\mathbf {#1}}}
\newcommand{\bbb}[1]{\ensuremath{\mathbb {#1}}}

\renewcommand{\phi}{\varphi}

\newcommand{\sq}[1]{\ensuremath{\langle#1\rangle}}

\newcommand{\restr}{\mathop{\upharpoonright}}

\newcommand{\notarrow}{\kern .42em\not\kern -.42em\longrightarrow}
\renewcommand{\th}{\ensuremath{{}^{\text{th}}}}
\newcommand{\sm}{{\text-}\!\!\sum}

\newcommand{\al}{\alpha}
\newcommand{\om}{\omega}
\newcommand{\vp}{\varphi}
\newcommand{\sse}{\subseteq}
\newcommand{\contains}{\supseteq}
\newcommand{\forces}{\Vdash}
 
\DeclareMathOperator{\ran}{ran}

\newcommand{\re}{\restriction}
\newcommand{\bP}{\mathbb{P}}

\newcommand{\ra}{\rightarrow}

\newcommand{\Lra}{\Longrightarrow}
\newcommand{\Llra}{\Longleftrightarrow}
\newcommand{\emptyse}{0}

\newcommand{\fin}{\mathrm{Fin}}
\newcommand{\Fin}{\mathrm{Fin}}

\newcommand{\A}{{\mathscr{A}}}
\newcommand{\F}{{\mathcal{F}}}
\newcommand{\GG}{{\mathcal{G}}}
\newcommand{\VV}{{\mathcal{V}}}
\newcommand{\Pset}{\mathcal{P}}
\newcommand{\restrict}{\mathord{\upharpoonright}}
\newcommand{\U}{{\mathcal{U}}}
\newcommand{\ZFC}{\mathrm{ZFC}}
\newcommand{\BB}{\mathcal{B}}
\newcommand{\cube}{{\left[\omega\right]}^{\omega}}
\newcommand{\V}{{\mathbf{V}}}

\newcommand{\XX}{{\mathcal{X}}}
\newcommand{\E}{{\mathcal{E}}}
\renewcommand{\c}{\mathfrak{c}}
\newcommand{\lc}{\left|}
\newcommand{\rc}{\right|}
\newcommand{\I}{{\mathcal{I}}}
\newcommand{\J}{{\mathcal{J}}}
\newcommand{\VP}{{\mathbf{V}}^{\bbb{P}}}

\newtheorem{claim}[thm]{Claim}

\theoremstyle{definition}
\newtheorem{defn}[thm]{Definition}
 
\theoremstyle{remark}

\newcommand{\noprint}[1]{\relax}

\title{The next best thing to a P-point} \author{Andreas Blass}
\address{Department of Mathematics\\
  University of Michigan\\
  Ann Arbor, MI 48109--1043, U.S.A.}  \email{ablass@umich.edu}
\urladdr{\url{http://www.math.lsa.umich.edu/~ablass}} \thanks{Blass
  was partially supported by NSF grant DMS-0653696.}  \author{Natasha
  Dobrinen}
\address{Department of Mathematics\\
  University of Denver \\
  2360 Gaylord St.\\ Denver, CO \ 80208 U.S.A.}
\email{natasha.dobrinen@du.edu}
\urladdr{\url{http://web.cs.du.edu/~ndobrine}} \thanks{Dobrinen was
  partially supported by the Fields Institute and the National Science
  Foundation during the Thematic Program in Forcing and Its
  Applications 2012, and by a Simons Foundation Collaborative Grant}
\author{Dilip Raghavan}
\address{Department of Mathematics\\
  National University of Singapore\\
  Singapore 119076} \email{raghavan@math.nus.edu.sg}
\urladdr{\url{http://www.math.toronto.edu/~raghavan}} \thanks{Raghavan
  was partially supported by a National University of Singapore
  research grant and by the Fields Institute during its Thematic
  Program on Forcing and Its Applications, 2012} 

\begin{document}

\begin{abstract}
We study ultrafilters on $\omega^2$ produced by forcing with the
quotient of $\scr P(\omega^2)$ by the Fubini square of the Fr\'echet
filter on $\omega$.  We show that such an ultrafilter is a weak P-point
but not a P-point and that the only non-principal ultrafilters strictly
below it in the Rudin-Keisler order are a single isomorphism class of
selective ultrafilters.  We further show that it enjoys the strongest
square-bracket partition relations that are possible for a
non-P-point.  We show that it is not basically generated but
that it shares with basically generated ultrafilters the property of
not being at the top of the Tukey ordering.  In fact, it is not
Tukey-above $[\omega_1]^{<\omega}$, and it has only continuum many
ultrafilters Tukey-below it.  A tool in our proofs is the
analysis of similar (but not the same) properties for ultrafilters
obtained as the sum, over a selective ultrafilter, of non-isomorphic
selective ultrafilters.  
\end{abstract}

\maketitle


\section{Introduction} \label{intro} 

Quotients of the form $\Pset(\omega) \slash \I$ where $\I$ is an
analytic ideal on $\omega$ are referred to as analytic quotients.
Analytic quotients have been well studied in the literature (see
\cite{mazur}, \cite{stevogap1}, and \cite{ilijasthesis}).  These
studies have usually focused on the structure of gaps in such
quotients or on lifting isomorphisms between $\Pset(\omega) \slash \I$
and $\Pset(\omega) \slash \J$, topics that are closely related.
$\Pset(\omega) \slash \I$ is a Boolean algebra, and hence is a notion
of forcing.\footnote{Strictly speaking, the notion of forcing is the
  Boolean algebra minus its zero element; such deletions or additions
  of zero elements will be tacitly assumed wherever necessary.}
If forcing with an analytic quotient $\Pset(\omega) \slash \I$ does
not add any new subset of $\omega$, then the generic filter it adds is
in fact an ultrafilter on $\omega$ that is disjoint from $\I$.
Ultrafilters added by analytic quotients have not been as extensively
investigated, except for the most familiar analytic quotient
$\Pset(\omega) \slash \Fin$, which adds a selective ultrafilter.
Recall the following definitions.
\begin{df}
  An ultrafilter \scr U on $\omega$ is \emph{selective} if, for every
  function $f:\omega\to\omega$, there is a set $A\in\scr U$ on which
  $f$ is either one-to-one or constant.  It is a \emph{P-point} if,
  for every $f:\omega\to\omega$, there is $A\in\scr U$ on which $f$ is
  finite-to-one or constant.
\end{df}
Indeed, selective ultrafilters can be completely characterized in
terms of genericity over $\Pset(\omega) \slash \Fin$ -- a well known
theorem of Todorcevic (\cite{carlosstevo}) states that in the presence
of large cardinals an ultrafilter on $\omega$ is selective if and only if it is
$\left( \mathbf{L}(\mathbb{R}), \Pset(\omega) \slash \Fin
\right)$-generic.  
Similar characterizations were recently shown for a large class of
ultrafilters forming a precise hierarchy above selective ultrafilters
(see \cite{Dobrinen/TodorcevicPart1}, \cite{Dobrinen/TodorcevicPart2},
and \cite{Mijares/Nieto}). 

The generic ultrafilter added by $\Pset(\omega) \slash \Fin$ has a
simple Rudin-Keisler type as well as a simple Tukey type.  Let $\F$ be
a filter on a set $X$ and $\GG$ a filter on a set $Y$.  Recall that we
say that $\F$ is \emph{Rudin-Keisler(RK) reducible to} $\GG$ or
\emph{Rudin-Keisler(RK) below} $\GG$, and we write $ \F \; {\leq}_{RK}
\; \GG$, if there is a map $f: Y \rightarrow X$ such that for each $a
\subseteq X$, $a \in \F$ iff ${f}^{-1}(a) \in \GG$.  $\F$ and $\GG$
are \emph{RK equivalent}, written $\scr F\equiv_{RK}\scr G$, if $\F \;
{\leq}_{RK} \; \GG$ and $\GG \; {\leq}_{RK} \; \F$.
If $\scr F$ and $\scr G$ are ultrafilters, then $\scr F\equiv_{RK}\scr
G$ if and only if there is a permutation $f:\omega\to\omega$ such that
$\scr F=\{a\subseteq\omega:f^{-1}(a)\in\scr G\}$.  For this reason,
ultrafilters that are RK equivalent are sometimes said to be \emph{(RK)
  isomorphic}.
It is well-known that selective ultrafilters are
minimal in the Rudin-Keisler ordering, meaning that any ultrafilter
that is RK below a selective ultrafilter is RK equivalent to that
selective ultrafilter.

The Tukey types of selective ultrafilters have also been completely
characterized.  We say that a poset $\langle D, \leq \rangle$ is
\emph{directed} if any two members of $D$ have an upper bound in $D$.
A set $X \subseteq D$ is \emph{unbounded in $D$} if it doesn't have an
upper bound in $D$.  A set $X \subseteq D$ is said to be \emph{cofinal
  in $D$} if $\forall y \in D\, \exists x \in X \left[y \leq x
\right]$.  Given directed sets $D$ and $E$, a map $f: D \rightarrow E$
is called a \emph{Tukey map} if the image of every unbounded subset of
$D$ is unbounded in $E$.  A map $g: E \rightarrow D$ is called a
\emph{convergent map} if the image of every cofinal subset of $E$ is
cofinal in $D$.  It is not difficult to show that there is a Tukey map
$f: D \rightarrow E$ if and only if there is a convergent $g: E
\rightarrow D$.  When this situation obtains, we say that $D$ is
\emph{Tukey reducible} to $E$, and we write $D \; {\leq}_{T} \; E$.
The relation ${\leq}_{T}$ is a quasi order, so it induces an
equivalence relation in the usual way: $D \; {\equiv}_{T} \; E$ if and
only if both $D \; {\leq}_{T} \; E$ and $E \; {\leq}_{T} \; D$ hold.
If $D \; {\equiv}_{T} \; E$, we say that $D$ and $E$ are \emph{Tukey
  equivalent} or have the same \emph{cofinal type}.  It is worth
noting that if $\kappa$ is an infinite cardinal and if $D$ is a
directed set of size $\kappa$, then $D \; {\leq}_{T} \; \langle
{\left[\kappa\right]}^{< \omega}, \subseteq \rangle$.  If $\U$ is any
ultrafilter, then $\langle \U, \supseteq \rangle$ is a directed set.
When ultrafilters are viewed as directed sets in this way, Tukey
reducibility is a coarser quasi order than RK reducibility.  It is
also worth noting that there is always an ultrafilter $\U$ on $\omega$
such that $\U \; {\equiv}_{T} \; \langle {\left[\c\right]}^{< \omega},
\subseteq \rangle$; every ultrafilter on $\omega$ is Tukey below this
$\U$.  In \cite{tukey}, it is shown that selective ultrafilters are
Tukey minimal.  Moreover, if $\U$ is selective and $\VV$ is any
ultrafilter that is Tukey below $\U$, then $\VV \; {\equiv}_{RK} \;
{\U}^{\alpha}$, for some $\alpha < {\omega}_{1}$ (see \cite{tukey} for
the definition of ${\U}^{\alpha}$).  Similar results have been proved
in \cite{Dobrinen/TodorcevicPart1} and \cite{Dobrinen/TodorcevicPart2}
for a large class of rapid p-points forming a hierarchy of
ultrafilters Tukey above selective ultrafilters.

On the other hand, the situation is not as clear for ultrafilters
added by other analytic quotients; neither their Rudin-Keisler type
nor their Tukey type has been well-studied.  In this paper we
investigate the generic ultrafilter added by a specific Borel
quotient.  In order to describe it, we introduce some preliminary
notation.
\begin{df} \label{df:misc} For $p \subseteq {\omega}^{2}$ and $n \in
  \omega$, let $p(n) = \{m \in \omega: \langle n, m \rangle \in p\}$.
  For $x \in {}^{\omega}{\left( \Pset(\omega) \right)}$ and $a
  \subseteq \omega$, $x \restrict a = \{\langle n, m \rangle \in
  {\omega}^{2}: n \in a \ \text{and} \ m \in x(n)\}$.  
According to these conventions, we notationally identify a set
$p\subseteq\omega^2$ with the sequence of its sections $p(n)$;
conversely, we shall sometimes notationally identify a sequence  $x \in
{}^{\omega}{\left( \Pset(\omega) \right)}$ and the corresponding set
$x\restr\omega$.
  ${\pi}_{1}$ is the projection of ${\omega}^{2}$ to the first
  co-ordinate and ${\pi}_{2}$ is projection to the second.  More
  formally, for $\langle n, m \rangle \in {\omega}^{2}$,
  ${\pi}_{1}(\langle n, m \rangle) = n$ and ${\pi}_{2}(\langle n, m
  \rangle) = m$.
 \end{df}
 \begin{df} \label{df:filtersum} Let $\U$ and $\scr V_n$ for $ n \in
   \omega$ be filters on $\omega$.
We define the \emph{$\U$-indexed
     sum of the ${\VV}_{n}$'s} as
 \begin{align*}
   \U \sm_n
   {\VV}_{n}=\{X\subseteq\omega^2:\{n\in\omega:X(n)\in{\VV}_{n}\}
   \in\U\}.
 \end{align*}
In the case when all the ultrafilters $\scr V_n$ are the same \scr V,
we write $\scr U\otimes\scr V$ for $\scr U\sm_n\scr V_n$.  If,
furthermore, $\scr V=\scr U$, then we abbreviate $\scr U\otimes\scr U$
as $\scr U^{\otimes2}$.

 For a set $A$, the \emph{Fr\'echet ideal on $A$} is $\{B \subseteq A:
 B \ \text{is finite}\}$ and the \emph{Fr\'echet filter on $A$} is
 $\{B \subseteq A: A - B \ \text{is finite}\}$.  $\F$ denotes the
 \emph{Fr\'echet filter} on $\omega$.

 For a filter $\GG$ on a set $X$, ${\GG}^{\ast}$ is the \emph{dual
   ideal to $\GG$} -- that is ${\GG}^{\ast} = \{X - a: a \in \GG\}$.
 ${\GG}^{+} = \Pset(X) - {\GG}^{\ast}$.
\end{df}
The ideal $\I = {\left( {\F}^{\otimes2} \right)}^{\ast}$ is a
${F}_{\sigma\delta\sigma}$ ideal, and it is easy to show that it is
not ${G}_{\delta\sigma\delta}$.  We will abuse notation and write the
quotient $\Pset({\omega}^{2}) \slash \I$ as $\Pset({\omega}^{2}) /
{\F}^{\otimes2}$.  Forcing with the Boolean algebra
$\Pset({\omega}^{2}) / {\F}^{\otimes2}$ is equivalent to forcing with
${\left( {\F}^{\otimes2} \right)}^{+}$ ordered by inclusion.  This
quotient does not add any new reals (see Lemma \ref{la:noreals}), and
hence adds an ultrafilter on ${\omega}^{2}$.  It is not hard to see
that this ultrafilter is not a P-point; indeed, the first projection
$\pi_1$ is neither finite-to-one nor constant on any set in this
ultrafilter (or even on any set in $(\scr F^{\otimes2})^+$).

In Section \ref{sec:gen1}, we study some combinatorial properties of
the generic ultrafilter added by $\Pset({\omega}^{2}) /
{\F}^{\otimes2}$.  We determine the extent to which the generic
ultrafilter possesses a certain partition property which has been
investigated by Galvin and Blass (\cite{mod-view}) before.  An
ultrafilter \scr W on a countable set $S$ is said to be \emph{$(n,
  h)$-weakly Ramsey}, where $n$ and $h$ are natural numbers, if for
every partition of ${\left[S\right]}^{n}$ into finitely many pieces,
there is $W \in \scr W$ such that ${\left[W\right]}^{n}$ intersects at
most $h$ pieces of the partition.  For $n \geq 2$, the property of
being $(n, 1)$-weakly Ramsey is equivalent to being selective.
Moreover, for each $n \geq 2$, there is a largest natural number $T(n)$ such
that, provably in ZFC, if an ultrafilter is $(n, T(n) - 1)$-weakly Ramsey, then it is a
P-point.
Theorem \ref{wk-part-gen} shows that the generic ultrafilter added by
$\Pset({\omega}^{2}) / {\F}^{\otimes2}$ is $(n, T(n))$-weakly Ramsey.
Thus the generic ultrafilter satisfies the strongest partition
property which a non-P-point is able to satisfy.  Furthermore, the
location of the generic ultrafilter in the Rudin-Keisler ordering is
fully determined.
It is shown that all non-principal ultrafilters
that are strictly RK below the generic ultrafilter $\GG$ are RK
equivalent to ${\pi}_{1}(\GG)$.  We also show in Section
\ref{sec:gen1} that the generic ultrafilter is a weak P-point.

In Sections \ref{sec:gen2} and \ref{sec:nicemaps} we prove
canonization theorems for monotone maps from the generic ultrafilter
to $\Pset(\omega)$.  Let $D$ and $E$ be directed sets.  A map $f: E
\rightarrow D$ is called \emph{monotone} if
\begin{align*}
  \forall {e}_{0}, {e}_{1} \in E \,\left[{e}_{0} \leq {e}_{1} \implies
    f({e}_{0}) \leq f({e}_{1})\right].
\end{align*} 
$f$ is said to be \emph{cofinal in $D$} if $\forall d \in D \exists e
\in E \left[d \leq f(e)\right]$.  It is clear that if $f$ is monotone
and cofinal in $D$, then $f$ is convergent.  It can be checked that if
$\U$ is an ultrafilter and $D$ is any directed set such that $\U \;
{\leq}_{T} \; D$, then there is a map from $D$ to $\U$ which is
monotone and cofinal in $\U$.  Therefore, understanding monotone maps
from the generic ultrafilter to $\Pset(\omega)$ is necessary to
analyze its Tukey type.

The first canonization of monotone maps from an ultrafilter to
$\Pset(\omega)$ was done for the basic ultrafilters by Dobrinen and
Todorcevic in \cite{dt}.  The crucial notion of a basic poset was
identified by Solecki and Todorcevic in \cite{basic}, who showed the
importance of this concept for the Tukey theory of definable ideals.
Dobrinen and Todorcevic~\cite{dt} proved that within the category of
ultrafilters on $\omega$, the basic posets are precisely the P-points.
Their canonization result, showing that  monotone maps on P-points are
continuous on some filter base,  allowed them to prove that 
there are only continuum many ultrafilters that are Tukey below any
fixed P-point.  They also introduced a weakened version of the
property of being basic, which they called being basically generated.
\begin{df} \label{df:base} Let $\U$ be an ultrafilter.  A set $\BB
  \subseteq \U$ is said to be a \emph{base} for $\U$ if $\forall a \in
  \U\, \exists b \in \BB\,\left[b \subseteq a\right]$.
\end{df}
\begin{df} \label{df:basicallygenerated} Let $\U$ be an ultrafilter on
  a countable set $X$.  We say that $\U$ is \emph{basically generated}
  if there is a base $\BB \subseteq \U$ with the property that for
  every $\langle {b}_{n}: n \in \omega \rangle \subseteq \BB$ and $b
  \in \BB$, if $\langle {b}_{n}: n \in \omega \rangle$ converges to
  $b$, then there exists $X \in \cube$ such that ${\bigcap}_{n \in
    X}{{b}_{n}} \in \U$.  Here convergence is with respect to the
  natural topology on $\Pset(X)$ induced by identifying $\Pset(X)$
  with the product space ${2}^{X}$, with $2 = \{0, 1\}$ having the
  discrete topology.
\end{df}
It is not hard to see that if $\U$ is basically generated, then
${\left[{\omega}_{1}\right]}^{< \omega} \; {\not\leq}_{T} \; \U$.
Dobrinen and Todorcevic~\cite{dt} pointed out that all the
ultrafilters obtained  by closing the class of P-points under the
operation of taking sums as in Definition \ref{df:filtersum} are
basically generated by a base that is also closed under finite
intersections.

In \cite{tukey} a canonization theorem was proved for monotone maps to
$\Pset(\omega)$ from any ultrafilter that is basically generated by a
base that is closed under finite intersections.  This canonization
result was used in \cite{tukey} to prove that there are only continuum
many ultrafilters that are Tukey below any fixed ultrafilter that is
basically generated by a base that is closed under finite
intersections.  The question of whether there are any ultrafilters
that are not above ${\left[ {\omega}_{1} \right]}^{< \omega}$ and also
not basically generated was left open in both \cite{dt} and
\cite{tukey}.

The canonization results in Sections \ref{sec:gen2} and
\ref{sec:nicemaps} imply that the generic ultrafilter added by
$\Pset({\omega}^{2}) / {\F}^{\otimes2}$ has only continuum many
ultrafilters Tukey below it.  In particular it is not of the maximal
possible Tukey type, ${\left[ \c \right]}^{< \omega}$.  This is
strengthened in Section \ref{sec:gen2} to prove that the generic
ultrafilter is not Tukey above ${\left[{\omega}_{1}\right]}^{<
  \omega}$.  The canonization obtained in Theorem \ref{thm:nicemaps}
is used in Theorem \ref{thm:canonical} to show that any Tukey
reduction from the generic ultrafilter $\GG$ to an arbitrary
ultrafilter $\VV$ can be replaced with a \emph{Rudin-Keisler}
reduction from an associated filter $\GG(P)$ to $\VV$.  This is an
exact analogue of Theorem 17 of \cite{tukey}, which was established
there for all ultrafilters that are basically generated by a base that
is closed under finite intersections.  These results show that the
generic ultrafilter added by $\Pset({\omega}^{2}) / {\F}^{\otimes2}$
has much in common with basically generated ultrafilters.  However, we
show in Section \ref{sec:notbasicallygenerated} that it fails to be
basically generated.  Thus this generic ultrafilter provides the first
known example of an ultrafilter that is not basically generated and is
not of the maximal possible Tukey type.
\begin{prob}
 Investigate the Rudin-Keisler and Tukey types of ultrafilters added
 by forcing with other analytic quotients.
\end{prob}

\section{Sums of Selective Ultrafilters} \label{sum}

Although the primary topic of this paper is the study of generic
ultrafilters obtained by forcing with $\scr P(\omega^2)/\scr
F^{\otimes2}$, this study will make much use of certain other
ultrafilters on $\omega^2$, namely sums of non-isomorphic selective
ultrafilters indexed by another selective ultrafilter.  The present
section is devoted to the study of such sums.  

The results in Sections~\ref{sum} and \ref{sec:gen1} are due to
Blass. Most of them were presented in a lecture at the Fields
Institute in September, 2012, but Theorem~\ref{wpp} was obtained later
in 2012.

\subsection{Definition and Basic Facts}

Throughout this section, $\scr U$ and $\scr V_n$ for $n\in\omega$ are
pairwise non-isomorphic selective ultrafilters on $\omega$.  Our
primary subject here will be the $\scr U$-indexed sum of the $\scr
V_n$'s, defined above as 
\[
\scr U\sm_n\scr V_n=\{X\subseteq\omega^2:\{n\in\omega:X(n)\in\scr
V_n\} \in\scr U\}.
\]
Recall that $X(n)$ denotes the $n\th$ vertical section of $X$,
$\{r:\sq{n,r}\in X\}$.  So a set $X$ is in $\scr U\sm_n\scr V_n$ if
and only if its vertical sections $X(n)$ are in the corresponding
ultrafilters $\scr V_n$ for \scr U-almost all $n$.  

\begin{rmk}
  The sum $\scr U\sm_n\scr V_n$ would be unchanged if we replaced the
  ultrafilters $\scr V_n$ by arbitrary other ultrafilters $\scr V'_n$
  for a set of $n$'s that is not in \scr U.  Thus, what we say in this
  section about pairwise non-isomorphic selective ultrafilters would
  remain true under the weaker assumption that \scr U is selective and
  that \scr U-almost all of the $\scr V_n$ are selective and not
  isomorphic to each other (or to \scr U).
\end{rmk}

The following lemma collects some elementary facts about sums of
ultrafilters.
The proofs are omitted because they amount to just
inspection of the definition.
\begin{la}      \label{pi1}
  \begin{lsnum}
\item $\scr U\sm_n\scr V_n$ is an ultrafilter on $\omega^2$.
 \item  $\scr U\sm_n\scr V_n$ contains the sets
  \[
\{\sq{x,y}\in\omega^2:n<x\text{ and }f(x)<y\}
\]
 for all $n\in\omega$
  and all $f:\omega\to\omega$. 
\item The projection $\pi_1:\omega^2\to\omega$ to the first factor
  sends $\scr U\sm_n\scr V_n$ to \scr U.
\item The projection $\pi_1$ is neither finite-to-one nor constant on
  any set in $\scr U\sm_n\scr V_n$.  Thus $\scr U\sm_n\scr V_n$ is not
  a P-point.
  \end{lsnum}
\end{la}

Part (2) of the lemma implies in particular that $\scr U\sm_n\scr V_n$
contains the above-diagonal set $\{\sq{x,y}:x<y\}$, which we can identify
with the set $[\omega]^2$ of two-element subsets of $\omega$. 

The preceding lemma used only that \scr U and the $\scr V_n$'s are
non-principal ultrafilters.  The next lemma, which describes how
$\pi_2$ acts, uses our assumptions of selectivity and non-isomorphism.
It requires somewhat more work, so we give the proof even though it
has long been known.

\begin{la}      \label{pi2}
  There is a set in $\scr U\sm_n\scr V_n$ on which the projection
  $\pi_2:\omega^2\to\omega$ to the second factor is one-to-one.  Thus,
  $\scr U\sm_n\scr V_n$ is isomorphic to its $\pi_2$-image, which is
  the limit, with respect to \scr U, of the sequence \sq{\scr
    V_n:n\in\omega}. 
\end{la}

\begin{pf}
We prove only the first statement, because the second, identifying the
image under $\pi_2$ with the limit in $\beta\omega$, is a
straightforward verification.

The first assertion will follow if we find pairwise disjoint sets
$A_n\subseteq\omega$ with each $A_n\in\scr V_n$, for then $\{\sq{n,y}:
n\in\omega\text{ and }y\in A_n\}$ is as required.  We obtain such sets
$A_n$ in a sequence of three steps.

For each $m\neq n$ in $\omega$, the ultrafilters $\scr V_m$ and $\scr
V_n$ are not isomorphic, so in particular they are distinct.  Thus,
there is a set $C_{n,m}\in\scr V_n$ that is not in $\scr V_m$.  

Because $\scr V_n$ is a P-point, it contains a set $B_n$ that is
almost included in $C_{n,m}$ for all $m\in\omega-\{n\}$.  Choose such
a $B_n$ for each $n$.  Thus, each $\scr V_n$ contains the $B_n$ with
the same subscript but the complements $\omega-B_m$ of all the other
$B_m$'s.  

Let 
\[
A_n=B_n\cap\bigcap_{m<n}(\omega-B_m).
\]
Then $A_n$ is the intersection of finitely many sets from $\scr V_n$,
so it is in $\scr V_n$.  For any $m<n$, $A_m$ and $A_n$ are clearly
disjoint, so the sets $A_n$ are as required.
\end{pf}

An amplification of this argument gives the following complete
classification, modulo $\scr U\sm_n\scr V_n$, of all functions
$f:\omega^2\to\omega$.

\begin{prop}
  If $f:\omega^2\to\omega$, then there is a set $X\in\scr U\sm_n\scr
  V_n$ such that $f\restr X$ is one of the following.
  \begin{ls}
    \item a constant function
\item $\pi_1$ followed by a one-to-one function
\item a one-to-one function
  \end{ls}
\end{prop}

\begin{pf}
Consider first, for each $n\in\omega$, the $n\th$ section of $f$, that
is, the function $f_n:\omega\to\omega:y\to f(n,y)$.  Since $\scr V_n$
is selective, $f_n$ is constant or one-to-one on some set $A_n\in\scr
V_n$.  Furthermore, as \scr U is an ultrafilter, it contains a set $B$
such that either $f_n$ is constant on $A_n$ for all $n\in B$ or $f_n$
is one-to-one on $A_n$ for all $n\in B$.

We treat first the case where $f_n$ is constant on $A_n$ for all $n\in
B$.  Then, on the set 
\[
Y=\{\sq{n,y}:n\in B\text{ and }y\in A_n\},
\]
which is in $\scr U\sm_n\scr V_n$, our function $f$ factors through
$\pi_1$; $f(x,y)=g(x)=g(\pi_1(x,y))$, where $g(n)$ is defined
as the constant value taken by $f_n$ on $A_n$.  

As \scr U is selective, $g$ is either constant or one-to-one on a set
$C\in\scr U$.  If $g$ is constant on $C$ then $f$ is constant on
$X=Y\cap{\pi_1}^{-1}(C)$, and we have the first alternative in the
proposition.   If, on the other hand, $g$ is one-to-one on $C$, then
the same $X$ gives us the second alternative in the proposition.

It remains to treat the case where $f_n$ is one-to-one on $A_n$ for
all $n\in B$.  In what follows, $n$ is intended to range only over
$B$.  Then $f_n(\scr V_n)$ is an ultrafilter isomorphic to $\scr V_n$,
so, as $n$ varies over $B$, these are pairwise non-isomorphic
selective ultrafilters.  

Arguing as in the proof of Lemma~\ref{pi2}, we find pairwise disjoint
sets $Z_n\in f_n(\scr V_n)$ for all $n\in B$.  We claim that $f$ is
one-to-one on
\[
X=\{\sq{n,y}:n\in B,\ y\in A_n,\ \text{and }f_n(y)\in Z_n\}.
\]
Indeed, if we had $\sq{n,y},\sq{n',y'}\in X$ and $f(n,y)=f(n',y')$,
i.e., $f_n(y)=f_{n'}(y')$, then this would be an element of $Z_n\cap
Z_{n'}$, so disjointness requires $n=n'$.  Furthermore, as $f_n$ is
one-to-one on $A_n$, we would have $y=y'$.  This completes the proof
of the claim that $f$ is one-to-one on $X$.  Also, since \scr U
contains $B$ and since $\scr V_n$ contains both $A_n$ and
${f_n}^{-1}(Z_n)$ for all $n\in B$, we have $X\in\scr U\sm_n\scr V_n$,
and so we have the third alternative in the proposition.
\end{pf}

\begin{cor}     \label{sum-Q}
  $\scr U\sm_n\scr V_n$ is a Q-point.
\end{cor}

\begin{pf}
  If $f:\omega^2\to\omega$ is finite-to-one, then the first two
  alternatives in the proposition are impossible (in view of part~(4)
  of Lemma~\ref{pi1}), and the only remaining alternative is that $f$
  is one-to-one on a set in $\scr U\sm_n\scr V_n$.
\end{pf}

\begin{cor}     \label{sum-RK}
  The only non-principal ultrafilters strictly below $\scr U\sm_n\scr
  V_n$ in the Rudin-Keisler ordering are the isomorphic copies of \scr
  U.
\end{cor}
\subsection{Weak Partition Properties and
  P-points}     \label{sq-bracket-P} 

The goal of this subsection and the next is to show that $\scr
U\sm_n\scr V_n$ satisfies the strongest ``square-bracket'' partition
properties that are possible for a non-P-point.  In the present
subsection, we explain square-bracket partition relations for
ultrafilters, and we show that some of these relations require the
ultrafilter to be a P-point.  In the next subsection, we shall show
that the strongest relations not covered by this result are satisfied
by $\scr U\sm_n\scr V_n$.

\begin{df}
  Let $n$ and $h$ be natural numbers, let $S$ be a countable set (for
  our purposes $S$ is usually $\omega$ or $\omega^2$), and let \scr W
  be an ultrafilter on $S$.  Then \scr W is $(n,h)$-\emph{weakly
    Ramsey} if, for every partition of $[S]^n$ into finitely many
  pieces, there is a set $W\in\scr W$ such that $[W]^n$ meets at
  most $h$ of the pieces.
\end{df}

\begin{rmk}
It would make no difference in this definition if, instead of
partitions into an arbitrary finite number of pieces, we referred only
to partitions into $h+1$ pieces, the first non-trivial case.  The
remaining cases would then follow by a routine induction on the number
of pieces.

What we have defined as $(n,h)$-weakly Ramsey is often expressed in
the partition calculus notation as
\[
S\to[\scr W]^n_{h+1},
\]
meaning that, if $[S]^n$ is partitioned into $h+1$ pieces then \scr W
contains a set $W$ such that $[W]^n$ misses a piece.  The square
brackets around \scr W in this notation are used to indicate that we
have only weak homogeneity, missing one piece.  Round brackets
conventionally denote full homogeneity, namely meeting only one
piece.  
\end{rmk}

When considering $(n,h)$-weak Ramseyness, we always
assume $n\geq2$ and $h\geq1$ to avoid trivialities.

Clearly, if we increase $h$ while keeping $n$ fixed, the property of
$(n,h)$-weak Ramseyness becomes weaker.

As mentioned in Section~\ref{intro}, for each $n\geq2$, the property
of being $(n,1)$-weakly Ramsey is equivalent to selectivity.  We
intend to show next that, as we increase $h$ to values greater than 2,
$(n,h)$-weak Ramseyness continues to imply that \scr W is a P-point,
until $h$ reaches a certain critical value $T(n)$, which we shall
compute.

To show that, for certain $n$ and $h$, all $(n,h)$-weakly Ramsey
ultrafilters are P-points, we shall consider ultrafilters \scr W that
are not P-points, and we shall describe $h+1$ different types of
$n$-element sets that must occur inside each $W\in\scr W$.  It will be
convenient to describe these types first in an abstract, formal way,
and only afterward to connect them with $n$-element sets.  

The following definition describes the abstract types.  

\begin{df}
  An $n$-type is a list of $2n$ variables, namely $x_i$ and $y_i$ for
  $1\leq i\leq n$, with either $<$ or $=$ between each consecutive
  pair in the list, such that 
  \begin{ls}
    \item the $y_i$'s occur in the list in increasing order of their
      subscripts,
\item equality signs can occur only between two $x_i$'s, and 
\item each $x_i$ precedes the $y_i$ that has the same subscript.
  \end{ls}
Two such lists are considered the same if they differ only by 
permuting $x_i$'s that are connected by equality signs.
\end{df}

\begin{ex}
  A fairly typical 7-type is 
\[
x_3=x_1<x_5<y_1<x_2=x_7<y_2<y_3<x_4<y_4<x_6<y_5<y_6<y_7.
\]
Interchanging $x_3$ with $x_1$ or interchanging $x_2$ with $x_7$ or
both would result in a new representation of the same 7-type.  Any
other rearrangements would result in a different 7-type (or in a list
that fails to be a 7-type).
\end{ex}

\begin{rmk}     \label{perorder}
  An equivalent definition of $n$-type is as a linear pre-ordering of
  the set of $2n$ $x_i$'s and $y_i$'s subject to the requirements that
  its restriction to the $y_i$'s is the strict ordering according to
  subscripts, that any equivalence class with more than one element
  must consist entirely of $x_i$'s, and that each $x_i$ strictly
  precedes the corresponding $y_i$.  
\end{rmk}

The number of $n$-types is the critical number $T(n)$ mentioned above,
the border between those $(n,h)$-weak Ramsey properties that imply
P-point and those that do not.  We shall later describe a more
effective means to compute $T(n)$, but first we develop the connection
between types and Ramsey properties of ultrafilters.

\begin{df}
  Let $n\geq2$, let $\tau$ be an $n$-type, let $f:\omega\to\omega$,
  and let $\bld a=\{a_1<a_2<\dots<a_n\}$ be an $n$-element subset of
  $\omega$ with its elements listed in increasing order.  We say that
  \bld a \emph{realizes} $\tau$ with respect to $f$ and that $\tau$ is
  the $f$-\emph{type of} \bld a if the equations and inequalities in
  $\tau$ become true when every $y_i$ is interpreted as $a_i$ and
  every $x_i$ is interpreted as $f(a_i)$.
\end{df}

\begin{rmk}     \label{preorder}
  Because of the restrictions in the definition of $n$-types, there
  can be $n$-element sets \bld a that have no $f$-type.  This could
  happen if $a_i=f(a_j)$ for some $i$ and $j$, because $n$-types
  cannot say that $y_i=x_j$.  It could also happen if $f(a_i)\geq a_i$
  for some $i$, because $x_i$ must strictly precede $y_i$ in any
  $n$-type.  
\end{rmk}

The following proposition is, as far as we know, due to Galvin but
never published.  A brief indication of the proof is given, along with
an explicit statement of the result for $n=24$ attributed to Galvin,
on page~85 of \cite{mod-view}.  Since that brief indication used
ultrapowers, we take this opportunity to give a complete and purely
combinatorial proof.

\begin{prop}    \label{galvin}
  Let \scr W be a non-principal ultrafilter on $\omega$ that is not a
  P-point, and let $f:\omega\to\omega$ be a function that is neither
  finite-to-one nor constant on any set in \scr W.  For any natural
  number $n$, every set $W\in\scr W$ includes $n$-element subsets
  realizing with respect to $f$ all $n$-types.
\end{prop}

\begin{pf}
Fix \scr W, $W$, $f$, and $n$ as in the proposition, and fix an $n$-type
$\tau$.  Let 
\[
B=\{v\in\omega: W\cap f^{-1}(\{v\}) \text{ is infinite}\}
\]
and note that $f^{-1}(B)\in\scr W$; indeed, $f$ is finite-to-one on
the intersection of $W$ with the complement of $f^{-1}(B)$, so this
complement cannot be in \scr W.  Note also that $B$ is infinite, for
otherwise $f$ would take only finitely many values on $W\cap
f^{-1}(B)$ and would therefore be constant on a set in \scr W.

Now we shall produce a set $\bld a\subseteq W$ realizing the given
$n$-type $\tau$.  We go through the list $\tau$ of variables $x_i$ and
$y_i$ in the order given by $\tau$, assigning a value to each variable
in turn.  At each step, we proceed as follows, assuming we have just
reached a certain variable and have assigned values to the earlier
variables in $\tau$.
\begin{ls}
\item If we have reached $x_i$ and there is an equality symbol immediately
  before it in $\tau$, say $x_j=x_i$,  then we assign to $x_i$ the
  same value that was already assigned to $x_j$.
\item If we have reached $x_i$ and there is no equality symbol
  immediately before it (either because there is $<$ there or because
  $x_i$ is first in the list $\tau$), then we assign to $x_i$ a value
  in $B$ that is larger than all of the (finitely many) values already
  assigned to other variables.
\item If we have reached $y_i$, then $x_i$ has already been assigned a
  value $v$; we assign to $y_i$ a value $a_i$ in $W\cap f^{-1}(\{v\})$ that is
  larger than all of the values already assigned to other variables.
\end{ls}
In the second and third cases, where a new, large value must be chosen, we
use the fact that $B$ and $W\cap f^{-1}(\{v\})$ (for $v\in B$) are
infinite, so sufficiently large values are available.  The values
$a_i$ assigned to the $y_i$'s are all in $W$, and they are in
increasing order.  They were chosen so that their images $f(a_i)$ are
the values assigned to the corresponding $x_i$'s.  Using these
observations and the fact that, whenever a new value was chosen, it
was larger than all previously chosen values, it is easy to see that
$\bld a=\{a_1<a_2<\dots<a_n\}$ realizes $\tau$, as required.
\end{pf}

\begin{df}
Let $T(n)$ denote the number of $n$-types.
\end{df}

\begin{cor}     \label{weak-to-P}
  Every $(n,T(n)-1)$-Ramsey ultrafilter is a P-point.
\end{cor}

\begin{pf}
  Suppose \scr W were a counterexample.  Let $f:\omega\to\omega$
  witness that \scr W is not a P-point.  Partition the $n$-element
  subsets of $\omega$ into $T(n)$ pieces $C_\tau$, one for each
  $n$-type $\tau$, by putting into $C_\tau$ all those $n$-element sets
  that realize $\tau$ with respect to $f$.  Sets that realize no type
  with respect to $f$ can be thrown into any of the pieces.  By the
  proposition, every set in \scr W has $n$-element subsets in all of
  the pieces.  So this partition is a counterexample to
  $(n,T(n)-1)$-weak Ramseyness.
\end{pf}

In the next subsection, we shall show that this corollary is optimal;
a non-P-point can be $(n,T(n))$-weakly Ramsey.  It is therefore of
some interest to understand $T(n)$, and we devote the rest of this
subsection to this finite combinatorial topic.

There is a simple recurrence relation, not for $T(n)$ but for a
closely related and more informative two-variable function $T(n,k)$
defined as the number of $n$-types $\tau$ that have exactly $k$
equivalence classes of $x_i$'s.  Here two $x_i$'s are called
equivalent\footnote{This is the same as the equivalence relation
  arising from the pre-order in Remark~\ref{perorder}.} if, between
them in $\tau$, there are only $=$ signs, not $<$ signs.  (It follows,
by the definition of types, that the only variables occurring between
these two $x$'s are other $x$'s, all in the same equivalence class.)
In general, $k$ can be as small as 1 (in just one $n$-type, namely
$x_1=\dots=x_n<y_1<\dots<y_n$) and as large as $n$ (if there are no
$=$ signs in $\tau$).  Clearly,
\[
T(n)=\sum_{k=1}^nT(n,k).
\]
We have the recursion relation
\[
T(n,k)=kT(n-1,k)+(n+k-1)T(n-1,k-1)
\]
for $n\geq k\geq1$.  To see this, consider an arbitrary $n$-type
$\tau$ with $k$ equivalence classes of $x$'s.  Let $\sigma$ be the
induced $n-1$-type.  That is, obtain $\sigma$ by deleting $x_n$ and
$y_n$ from $\tau$ and combining the $=$ and $<$ signs around $x_n$ in
the obvious way.  (If both are $=$, use $=$, and otherwise use $<$.
Note that no combining is needed around $y_n$ since it is at the end
of the list $\tau$.  Also, $x_n$ could be at the beginning of $\tau$,
in which case it too would need no combining.)  The number of
equivalence classes in $\sigma$ is either $k$ or $k-1$, depending on
whether $x_n$ was, in $\tau$, equivalent to another $x_i$ or not.
Each of the $T(n-1,k)$ possible $\sigma$'s of the first sort arises
from $k$ possible $\tau$'s, because $x_n$ could have been in any of
the $k$ equivalence classes. This accounts for the first term on the
right side of our recurrence.  Each of the $T(n-1,k-1)$ possible
$\sigma$'s of the second sort arises from $n+k-1$ possible $\tau$'s,
because $x_n$ can be put into any of the intervals determined by the
$n-1$ $y$'s and $k-1$ equivalence classes of $x$'s in $\sigma$.  Here
the possible intervals include the degenerate ``intervals'' at the
left and right ends of $\sigma$, so the number of intervals is
$n+k-1$.  This accounts for the second term on the right side and thus
completes the proof of our recurrence relation.

There are other combinatorial interpretations of $T(n)$.  For example,
it is the number of rooted trees with $n+1$ labeled leaves, subject to
the requirement that every non-leaf vertex must have at least two
children, i.e., there is genuine branching at each internal node.  The
number of internal nodes can be as small as 1 (if all the leaves are
children of the root) and as large as $n$ (if all branching is
binary).  $T(n,k)$ counts these trees according to the number $k$ of
internal nodes.

For more information, tables of values, and references, see
\cite{oeis}, where $T(n)$ is (up to a shift of the indexing)
sequence~A000311 and $T(n,k)$ is (rearranged into a single sequence)
A134991. 

\subsection{Weak Partition Properties and Sums of Selective
  Ultrafilters}         

The goal of this subsection is to prove that ultrafilters of the form
$\scr U\sm_n\scr V_n$ (still subject to this section's convention that
\scr U and the $\scr V_n$ are pairwise non-isomorphic selective
ultrafilters) have the strongest weak-Ramsey properties that
Corollary~\ref{weak-to-P} permits for a non-P-point.  

\begin{thm}     \label{T-wk}
  $\scr U\sm_n\scr V_n$ is $(n,T(n))$-weakly Ramsey.
\end{thm}

\begin{pf}
  We shall need to refer to the types of $n$-element subsets of the
  set $\omega^2$ underlying our ultrafilter $\scr U\sm_n\scr V_n$.
  Types were defined above for $n$-element subsets of $\omega$, not
  $\omega^2$.  In principle, this is no problem, because we know, from
  Lemma~\ref{pi2}, that $\pi_2$ is an isomorphism from $\scr
  U\sm_n\scr V_n$ to an ultrafilter \scr W on $\omega$, so we can just
  transfer, via this isomorphism, any concepts and facts that we need.
  In practice, though, the need to repeatedly apply this isomorphism
  makes statements unpleasantly complicated, so we begin by
  reformulating the relevant facts about types in a way that lets us
  use $\scr U\sm_n\scr V_n$ directly.

  Our discussion of types above was relative to a fixed function $f$
  witnessing that the ultrafilter is not a P-point.  We know, from
  Lemma~\ref{pi1}, that $\pi_1$ is such a function for $\scr
  U\sm_n\scr V_n$; therefore $f=\pi_1\circ{\pi_2}^{-1}$ is such a
  function for $\scr W=\pi_2(\scr U\sm_n\scr V_n)$.  Notice, in this
  connection, that although $\pi_2$ is not globally one-to-one, it has
  a one-to-one restriction to a suitable set in $\scr U\sm_n\scr V_n$
  (Lemma~\ref{pi2}); enlarging this set slightly if necessary, we get
  a set $G\in\scr U\sm_n\scr V_n$ on which $\pi_2$ is bijective.
  By ${\pi_2}^{-1}:\omega\to G$ we mean the inverse of this bijection.

Now any $n$-element subset \bld a of $G$ corresponds to an $n$-element
subset $\bld b=\pi_2(\bld a)$ of $\omega$.  With
$f=\pi_1\circ{\pi_2}^{-1}$ as above, the $f$-type of \bld b will be
called simply the \emph{type} of \bld a.  Unraveling the definitions,
we find that this type can be described as follows, directly in terms
of \bld a.  Enumerate the elements of \bld a as
\[
\sq{p_1,q_1},\sq{p_2,q_2},\dots,\sq{p_n,q_n}
\]
in order of increasing second components, so $q_1<q_2<\dots<q_n$.
Then \bld a realizes the $n$-type $\tau$ if and only if all the
equations and inequalities in $\tau$ become true when the variables
$x_i$ are interpreted as $p_i$ and the $y_i$ are interpreted as $q_i$.

Let us use, for points in $\omega^2$, the familiar terminology
``$x$-coordinate'' and ``$y$-coordinate'' for the first and second
components.  Then, the $x$ and $y$ variables in a type represent the
$x$ and $y$ coordinates of the points of a set realizing the type.

Transferring Proposition~\ref{galvin} and its corollaries via the
isomorphism $\pi_2$, we learn that every set in $\scr U\sm_n\scr V_n$
contains $n$-element subsets realizing all $T(n)$ of the $n$-types.
The proof of the theorem will consist of showing that this is all one
can say in the direction of existence of many different kinds of
$n$-element subsets in all sets from $\scr U\sm_n\scr V_n$.
 
Consider an arbitrary partition $\Pi$ of the set $[\omega^2]^n$ of 
$n$-element subsets of $\omega^2$ into a finite number $z$ of pieces.
We shall show that there is a set $W\in\scr U\sm_n\scr V_n$ such that
\begin{ls}
  \item every $n$-element subset of $W$ realizes an $n$-type, and
\item any two subsets realizing the same $n$-type are in the same
  piece of the partition $\Pi$.
\end{ls}
This will clearly suffice to prove the theorem.  

Furthermore, we can treat the various types separately.  That is, it
suffices to find 
\begin{ls}
  \item a set $X\in\scr U\sm_n\scr V_n$ such that all $n$-element
    subsets of $X$ realize $n$-types and
\item for each $n$-type $\tau$, a set $Y_\tau\in\scr U\sm_n\scr V_n$
  such that all $n$-element subsets of $Y_\tau$ that realize $\tau$
  lie in the same piece of the partition $\Pi$.
\end{ls}
Indeed, the intersection of $X$ and all $T(n)$ of the sets $Y_\tau$'s
will then be a set $W$ as required above.

Let us first produce the required $X$ all of whose $n$-element subsets
realize types.  Inspecting the descriptions of types and realization,
we find that $X$ needs to have the following properties, in addition
to being in $\scr U\sm_n\scr V_n$.
\begin{lsnum}
  \item No two distinct elements of $X$ have the same $y$-coordinate
    (so that the $y_i$ can all be properly ordered in the type).
\item Each element of $X$ has its $x$-coordinate smaller than its
  $y$-coordinate (so that each $x_i$ can precede $y_i$ in the type).
\item No $x$-coordinate of a point in $X$ equals the $y$-coordinate of
  another point in $X$ (so that the type has no equalities between any
  $x_i$ and $y_j$).  
\end{lsnum}
Furthermore, these three requirements can be treated independently; if
we find a different $X\in\scr U\sm_n\scr V_n$ for each one, we can
just intersect those three $X$'s to complete the job.  For the first
two requirements, we already have the necessary $X$'s.  The first says
that $\pi_2$ is one-to-one on $X$, so Lemma~\ref{pi2} gives what we
need.  The second is satisfied by $\{\sq{x,y}\in\omega^2:x<y\}$ which
we noted, right after Lemma~\ref{pi1}, is in $\scr U\sm_n\scr V_n$.
We therefore concentrate on the third requirement.

For this requirement, it suffices to find a set $B\in\scr U$ that is
in none of the $\scr V_n$, for then $B\times(\omega-B)$ serves as the
required $X$.  Since \scr U is distinct from all the $\scr V_n$, we
can find, for each $n$, a set $B_n\in\scr U-\scr V_n$.  Then, since
\scr U is a P-point, we can find a single $B\in\scr U$ almost included
in each $B_n$.  Since each $\scr V_n$ is non-principal and doesn't
contain $B_n$, it cannot contain $B$ either.

This completes the proof that $\scr U\sm_n\scr V_n$ contains a set $X$
all of whose $n$-element subsets represent types.  It remains to
consider an arbitrary type $\tau$ and find a set $Y_\tau\in\scr
U\sm_n\scr V_n$ all of whose $n$-element subsets of type $\tau$ are in
the same piece of our partition $\Pi$.

Fix, therefore, a particular $\tau$ for the remainder of the proof.
It suffices to find an appropriate set $Y_\tau$ in the case that $\Pi$
is a partition into only two pieces.  The general case where $\Pi$ has
any finite number $z$ of pieces then follows.  Just consider all the
2-piece partitions coarser than $\Pi$, find an appropriate $Y$ for
each of these, and intersect these finitely many $Y$'s.  So from now
on, in addition to working with  a fixed $\tau$, we work with a fixed
partition $\Pi=\{R,[\omega^2]^n-R\}$ of $[\omega^2]^n$ into two pieces.  Our
goal is to find a set $Y\in\scr U\sm_n\scr V_n$ such that either all
$n$-element subsets of $Y$ realizing $\tau$ are in $R$ or none of them
are.  

Fortunately, a stronger result than this was already proved as
Theorem~7 in \cite[page~236]{sel-hom}.  We shall quote that result and
then indicate how it implies what we need here.

\begin{prop}[\cite{sel-hom}, Theorem~7] \label{analytic}
  Let there be given selective ultrafilters $\scr W_s$ on $\omega$ for
  all finite subsets $s$ of $\omega$.  Assume that, for all
  $s,t\in[\omega]^{<\omega}$, the ultrafilters $\scr W_s$ and $\scr
  W_t$ are either equal or not isomorphic.  Let there also be given an
  analytic subset \scr X of the set $[\omega]^\omega$ of infinite
  subsets of $\omega$.  Then there is a function $Z$ assigning, to
  each ultrafilter \scr W that occurs among the $\scr W_s$'s, some
  element $Z(\scr W)\in\scr W$ such that \scr X contains all or none
  of the infinite subsets $\{z_0<z_1<z_2<\dots\}$ that satisfy $z_n\in
  Z(\scr W_{\{z_0,\dots,z_{n-1}\}})$ for all $n\in\omega$.
\end{prop}

It is important here that $Z$ assigns sets $Z(\scr W)$ to ultrafilters
\scr W, not to their occurrences in the system \sq{\scr
  W_s:s\in[\omega]^{<\omega}}.  That is, if the same ultrafilter \scr W
  occurs as $\scr W_s$ for several sets $s$, then the same $Z(\scr W)$
  is used for all these occurrences of \scr W.

  There is an essentially unique reasonable way to regard our
  partition $\Pi$ of $[\omega^2]^n$, restricted to sets of type $\tau$,
  as a clopen (and therefore analytic) partition of $[\omega]^\omega$
  and to choose ultrafilters $\scr W_s$ among our \scr U and $\scr
  V_n$'s, so that Proposition~\ref{analytic} completes the proof of
  Theorem~\ref{T-wk}.  We spell this out explicitly in what follows.

For our fixed type $\tau$, we say that a variable $x_i$ or $y_i$ is in
\emph{position} $\alpha$ if it is preceded in $\tau$ by exactly
$\alpha$ occurrences of $<$.  The values of $\alpha$ that can occur
here range from 0 to $n+k-1$, where $k$ is the number of equivalence
classes of $x_i$'s as in our discussion of the recurrence for $T(n,k)$
in subsection~\ref{sq-bracket-P}.  Notice that the variables in any
particular position are either an equivalence class of $x$'s or a
single $y$.  

Given any finite or infinite increasing sequence $\vec z$ of natural
numbers, say $z_0<z_1<\dots<z_l\,(<\dots)$, we associate to
it an assignment of values to the variables in some initial segment of
$\tau$ by giving the value $z_\alpha$ to the variable(s) in position
$\alpha$.  If the sequence $\vec z$ is longer than $n+k$ then the
terms from $z_{n+k}$ on have no effect on this assignment.  If
$\vec z$ is shorter than $n+k$ then not all of the variables in $\tau$
receive values.  

For an infinite increasing sequence $\vec z$ (and indeed also for
finite $\vec z$ of length $\geq n+k$), we obtain in this way values,
say $p_i$, for all the $x_i$'s and values, say $q_i$, for all the
$y_i$'s ($1\leq i\leq n$), and we combine these values to form an
$n$-element subset of $\omega^2$:
\[
\bld a(\vec z)=\{\sq{p_1,q_1},\sq{p_2,q_2},\dots,\sq{p_n,q_n}\}.
\]
Note that the pairs \sq{p_i,q_i} occur in this list in order of
increasing $y$-coordinates (because $\tau$ is a type and $\vec z$ is
increasing) and that $\bld a(\vec z)$ has type $\tau$.  We define \scr
X to consist of those infinite subsets of $\omega$ whose increasing
enumeration $\vec z$ has $\bld a(\vec z)$ in the piece $R$ of our
partition $\Pi$.  Thus, membership of a set $A$ in \scr X depends only
on the first $n+k$ elements of $A$.  In particular, \scr X is clopen,
and therefore certainly analytic, in $[\omega]^\omega$.

To finish preparing an application of Proposition~\ref{analytic}, we
define ultrafilters $\scr W_s$ for all finite $s\subseteq\omega$ as
follows.  Let $\vec z$ be the increasing enumeration of $s$, and use
it as above to assign values to the variables in the positions 0
through $\min\{|s|,n+k\}-1$ in $\tau$.
\begin{lsnum}
\item If $|s|<n+k$ and the variables in position $|s|$ (the first
  variables not assigned values) are an equivalence class of $x_i$'s,
  then let $\scr W_s=\scr U$.
\item If $|s|<n+k$ and the variable in position $|s|$ is $y_i$, then
  $x_i$, occurring earlier in $\tau$, has been assigned a value $v$.
  Let $\scr W_s=\scr V_v$.
\item If $|s|\geq n+k$ then set $\scr W_s=\scr U$.  (This case is
  unimportant; we could use any selective ultrafilters here as long as
  we satisfy the ``equal or not isomorphic'' requirement in
  Proposition~\ref{analytic}.) 
\end{lsnum}

Since our \scr X is analytic and our $\scr W_s$ are selective
ultrafilters every two of which are either equal or not isomorphic,
Proposition~\ref{analytic} applies and provides us with sets $Z(\scr
U)$ and $Z(\scr V_r)$ for all $r\in\omega$ such that $\scr X$ contains
all or none of the infinite sets whose increasing enumerations $\vec
z=\sq{z_0<z_1<\dots}$ have $z_\alpha\in Z(\scr
W_{\{z_0,\dots,z_{\alpha-1}\}})$ for all $\alpha$.  It remains to
untangle the definitions and see what this homogeneity property
actually means.  Since membership of an infinite set in \scr X depends
only on that set's first $n+k$ members, the homogeneity property is
really not about infinite sets but about $(n+k)$-element sets and
their increasing enumerations $\vec z$.  We look separately at the
assumption and the conclusion in the homogeneity statement of
Proposition~\ref{analytic}.

  The assumption is that $z_\alpha\in Z(\scr
  W_{\{z_0,\dots,z_{\alpha-1}\}})$, which means that, when we use $\vec
    z$ to give values to the variables in $\tau$, all the values given
    to the $x_i$'s are in $Z(\scr U)$ and the value given to any $y_i$
    is in $Z(\scr V_v)$, where $v$ is the value given to the
    corresponding $x_i$.  This means exactly that $\bld a(\vec z)$ has
    all its elements in the set
\[
Y=\{\sq{p,q}:p\in Z(\scr U) \text{ and } q\in Z(\scr V_p)\}.
\]
Conversely, any $n$-element subset of type $\tau$ in this $Y$, listed
in increasing order of $y$-coordinates, is $\bld a(\vec z)$ for an
enumeration $\vec z$ satisfying the assumption in
Proposition~\ref{analytic}.

Because each $Z(\scr W)$ is in the corresponding ultrafilter \scr W,
we have $Y\in\scr U\sm_r\scr V_r$.  

The conclusion in the homogeneity statement is that all or none of the
sets satisfying the hypothesis are in \scr X.  This means, in view of
our choice of \scr X, that all or none of the associated $\bld a(\vec
z)$ are in $R$.

Collecting all this information, we have that $R$ contains all or none
of the $n$-element subsets of type $\tau$ in $Y$.  Thus, $Y$ is as
required to complete the proof of Theorem~\ref{T-wk}.
\end{pf}

\begin{cor}     \label{non-uf}
Let $X\in(\scr F^{\otimes2})^+$, and let $[X]^n$ be partitioned into
finitely many pieces.  Then there is $Y\subseteq X$ such that
$Y\in(\scr F^{\otimes2})^+$ and 
\begin{ls}
  \item every $n$-element subset of $Y$ realizes an $n$-type and
\item for each $n$-type $\tau$, all the $n$-element subsets of $Y$
  that realize it are in the same piece of the given partition.
\end{ls}
In particular, $[Y]^n$ meets at most $T(n)$ pieces of the given
partition.
\end{cor}

\begin{pf}
  Suppose for a moment that the continuum hypothesis holds, so that
  there are $2^{\ger c}$ non-isomorphic selective ultrafilters
  containing any given infinite subset of $\omega$.  Let
  $B=\{n\in\omega:X(n)\text{ is infinite}\}$.  The assumption that
  $X\in(\scr F^{\otimes2})^+$ means that $B$ is infinite.  Choose
  non-isomorphic selective ultrafilters \scr U and $\scr V_n$ for all
  $n\in\omega$ so that $B\in\scr U$ and $X(n)\in\scr V_n$ for all
  $n\in B$.  Then $X\in\scr U\sm_n\scr V_n$.  Extend the given
  partition arbitrarily to all of $[\omega^2]^n$ and invoke the proof
  of Theorem~\ref{T-wk} to get a set $Y\in\scr U\sm_n\scr V_n$ in
  which every $n$-element subset realizes a type and different subsets
  realizing the same type are in the same piece of our partition.
  Intersecting $Y$ with $X$, we get a set, still in $\scr U\sm_n\scr
  V_n$ and thus in $(\scr F^{\otimes2})^+$, with the required
  homogeneity property for the given partition of $[X]^n$.  This
  completes the proof under the assumption of the continuum
  hypothesis.

There are (at least) two ways to show that the result remains true in
the absence of the continuum hypothesis.  One way is to pass to a forcing
extension that has no new reals but satisfies the continuum
hypothesis.  The corollary holds in the extension, but it is only
about reals, so it must have held in the ground model.  The other way
is note that the  corollary is a $\Pi^1_2$ statement, so, by
Shoenfield's absoluteness theorem, it is absolute between the whole
universe $V$ and the constructible sub-universe $L$.  Since $L$
satisfies the continuum hypothesis, the corollary holds there and
thus also holds in $V$.
\end{pf}

\section{Ultrafilters Generically Extending $\scr F\otimes\scr F$}  
\label{sec:gen1}

We turn now to the main subject of this paper, ultrafilters on
$\omega^2$ produced by forcing with $\bbb P=(\scr F^{\otimes 2})^+$
partially ordered by inclusion.  Recall that we write $\scr F$ for the
Fr\'echet filter, the set of cofinite subsets of $\omega$.  Its Fubini
square $\scr F^{\otimes2}$ is the filter on $\omega^2$ consisting of
sets $X$ such that the sections $X(n)$ are cofinite for cofinitely
many $n$.  A convenient basis for $\scr F^{\otimes2}$ is the family of
``wedges''
\[
\{\sq{x,y}\in\omega^2:x>n\text{ and }y>f(x)\}.
\]
where $n$ ranges over natural numbers and $f$ ranges over functions
$\omega\to\omega$.

Our forcing conditions in \bbb P are the sets of
positive measure with respect to $\scr F^{\otimes2}$; these are the
sets $X$ whose sections $X(n)$ are infinite for infinitely many $n$.

The partially ordered set \bbb P is not separative.  Its separative
quotient is obtained by identifying any two conditions that have the
same intersection with a set in $\scr F^{\otimes2}$.  So the
separative quotient is the Boolean algebra
$\scr P(\omega^2)/\scr F^{\otimes2}$.

We call a forcing condition $X\in\bbb P$ \emph{standard} if every
nonempty section $X(n)$ is infinite.  Every condition can be shrunk to
a standard one by simply deleting its finite sections; the shrunk
condition is equivalent, in the separative quotient, to the original
condition.  We can therefore assume, without loss of generality, that
we always work with standard conditions. 

\begin{la} \label{la:noreals}
  The separative quotient of \bbb P is countably closed.  In
  particular, forcing with \bbb P adds no new reals.
\end{la}

\begin{pf}
  The second assertion of the lemma is a well-known consequence of the
  first, so we just verify the first.  Given a decreasing
  $\omega$-sequence in $\scr P(\omega^2)/\scr F^{\otimes2}-\{0\}$, we
  can choose representatives in \bbb P, say $A_0\supseteq
  A_1\supseteq\dots$, where we have arranged actual inclusion (rather
  than inclusion modulo $\scr F^{\otimes2}$) by intersecting each set
  in the sequence with all its predecessors.  Each $A_n$, being in
  \bbb P, has infinitely many infinite sections, so we can choose
  numbers $x_0<x_1<\dots$ such that $A_n(x_n)$ is infinite for every
  $n$.  Let $B=\{\sq{x_n,y}:n\in \omega\text{ and }y\in A_n\}$.  Then
  $B\in\bbb P$ and $B$ is included, modulo $\scr F^{\otimes2}$, in
  each $A_n$.  (In fact, it's included modulo $\scr
  F\otimes\{\omega\}$.)
\end{pf}

In what follows, we work with a ground model $V$ and its forcing
extensions.  Depending on the reader's preferences, $V$ can be the
whole universe of sets, in which case its forcing extensions are
Boolean valued models; or $V$ can be a countable transitive model of
ZFC, in which case forcing extensions are also countable transitive
models; or $V$ can be an arbitrary model of ZFC, in which case its
forcing extensions are again merely models of ZFC.  What we do below
is valid in any of these contexts.  Unless the contrary is specified,
``generic'' means generic over $V$.

\begin{cor} \label{cor:genericult}
Any generic filter $\scr G\subseteq\bbb P$ is, in the forcing
extension $V[\scr G]$, an ultrafilter on $\omega^2$.
\end{cor}

\begin{pf}
Since every subset $X$ of $\omega$ in $V[\scr G]$ is in $V$, it
suffices to notice that, for each such $X$, the set $\{A\in\bbb
P:A\subseteq X\text{ or }A\subseteq\omega-X\}$ is in $V$ and dense in
\bbb P, so it meets \scr G.
\end{pf}

In view of the corollary, we have written \scr G for the generic
filter, rather than the more customary $G$, since we have been using
script letters for ultrafilters.

The rest of this section is devoted to establishing the basic
combinatorial properties of \bbb P-generic ultrafilters \scr G.  Many
of these properties are the same as what we established in the
preceding section for selective-indexed sums of non-isomorphic
selective ultrafilters.  In particular, \scr G is not a P-point, and
it satisfies the strongest weak-Ramsey property compatible with not
being a P-point.  But \scr G also differs in an essential way from the
sums considered earlier, in that it is not a limit of countably many
other non-principal ultrafilters.  

\subsection{Basic Properties}

In this subsection, we collect some of the basic facts about \bbb
P-generic ultrafilters \scr G on $\omega^2$.  From now on, \scr G will
always denote such an ultrafilter and $\dot{\scr G}$ its canonical
name in the forcing language.

\begin{la}      \label{force-in-G}
For $A\in\bbb P$ and $X\subseteq\omega^2$, we have $A\forces
X\in\dot{\scr G}$ if and only if $A\subseteq X$ modulo $\scr
F^{\otimes2}$.  The ultrafilter \scr G is an extension of $\scr
F^{\otimes2}$. 
\end{la}

\begin{pf}
The first assertion of the lemma is a consequence of the well-known
facts that forcing is unchanged when a notion of forcing is replaced
by its separative quotient and that, for a separative notion of
forcing, $p$ forces $q$ to belong to the canonical generic filter if
and only if $p\leq q$.  

It follows from the first assertion that $\scr G \subseteq(\scr
F^{\otimes2})^+$.  Since \scr G is an ultrafilter, this is equivalent
to $\scr G\supseteq\scr F^{\otimes 2}$.
\end{pf}

\begin{cor} \label{cor:notp}
The first projection $\pi_1:\omega\to\omega$ is not finite-to-one or
constant on any set in \scr G.  Thus, \scr G is not a P-point.
\end{cor}

\begin{pf}
  $\pi_1$ is not finite-to-one or constant on any set in $(\scr
  F^{\otimes2})^+$.
\end{pf}

\begin{prop}    \label{p1-gen}
The image $\scr U=\pi_1(\scr G)$ of \scr G under the first projection
is $V$-generic for $\scr P(\omega)/\scr F$.  In particular, it is a
selective ultrafilter on  $\omega$.
\end{prop}

\begin{pf}
  Consider the function $c:[\omega]^\omega\to\bbb P$ that sends each
  infinite $X\subseteq\omega$ to the cylinder over it in $\omega^2$,
  i.e., $c(X)={\pi_1}^{-1}(X)$.  We claim that this function is a
  complete embedding of the forcing notion
  $([\omega]^\omega,\subseteq)$ into \bbb P.  In the first place, $c$
  clearly preserves the ordering relation $\subseteq$.  It also
  preserves incompatibility; if $X$ and $Y$ are incompatible in
  $[\omega]^\omega$, i.e., if their intersection is finite, then the
  intersection of $c(X)$ and $c(Y)$ is not in $(\scr F^{\otimes2})^+$,
  so $c(X)$ and $c(Y)$ are also incompatible.  Finally, if \scr A is a
  maximal antichain in $[\omega]^\omega$, then the antichain
  $\{c(A):A\in\scr A\}$ is maximal in \bbb P.  To see this, consider
  an arbitrary $X\in\bbb P$.  We intend to show that it is compatible
  with $c(A)$ for some $A\in\scr A$.  Let $B=\{n\in\omega:X(n)\text{
    is infinite}\}$.  Then $B$ is infinite.  By maximality of \scr A,
  find $A\in\scr A$ such that $A\cap B$ is infinite.  Then $X\cap
  c(A\cap B)$ is in \bbb P and is an extension of both $X$ and $c(A)$,
  as required.  This completes the proof that $c$ is a complete
  embedding.

By a well-known general fact about complete embeddings, it follows
that $c^{-1}(\scr G)$ is a $V$-generic subset of $[\omega]^\omega$.
But $c^{-1}(\scr G)$ is exactly $\pi_1(\scr G)$, so the first assertion
of the lemma is proved.  The second follows, as it it well-known that
the generic object adjoined by forcing with $[\omega]^\omega$ is a
selective ultrafilter on $\omega$.
\end{pf}

\subsection{Partition Property and Consequences}

\begin{thm}     \label{wk-part-gen}
  The generic ultrafilter \scr G is $(n,T(n))$-weakly Ramsey for every
  $n$.  In more detail, for every partition $\Pi$ of $[\omega^2]^n$
  into finitely many pieces, there is a set $H\in\scr G$ such that
  \begin{ls}
    \item every $n$-element subset of $H$ realizes an $n$-type, and
\item for each $n$-type $\tau$, all the $n$-element subsets of $H$
  that realize $\tau$ are in the same piece of $\Pi$.
  \end{ls}
\end{thm}

\begin{pf}
 The first assertion of the the theorem follows from the second
 because there are only $T(n)$ $n$-types.  The second assertion, in
 turn, follows immediately from Corollary~\ref{non-uf}, which says
 that conditions $H$ with the desired properties are dense in \bbb P,
 so the generic \scr G must  contain such an $H$.
\end{pf}

Note that the $T(n)$ in the theorem is optimal, because of
Corollary~\ref{weak-to-P} and the fact that \scr G is not a P-point.  

The next two corollaries could be proved by direct density arguments,
but it seems worthwhile to point out how they follow from the
partition properties in Theorem~\ref{wk-part-gen}.

\begin{cor}
  The second projection $\pi_2$ is one-to-one on a set in \scr G.
\end{cor}

\begin{pf}
  By Theorem~\ref{wk-part-gen}, let $H\in\scr G$ be a set all of whose
  2-element subsets realize types.  Then $\pi_2$ is one-to-one on $H$
  because no 2-type can have $y_1=y_2$.
\end{pf}

\begin{cor}
  If $f:\omega^2\to\omega$, then there is a set $X\in\scr G$ such that
  $f\restr X$ is one of the following.
  \begin{ls}
    \item a constant function
\item $\pi_1$ followed by a one-to-one function
\item a one-to-one function
  \end{ls}
\end{cor}

\begin{pf}
  Partition the set $[\omega^2]^2$ into two pieces by putting
  $\{\sq{p_1,q_1},\sq{p_2,q_2}\}$ into the first piece if
  $f(p_1,q_1)=f(p_2,q_2)$ and into the second piece if $f(p_1,q_1)\neq
  f(p_2,q_2)$.  Let $H\in\scr G$ be as in Theorem~\ref{wk-part-gen}
  for $n=2$ and this partition.  

  Consider first those 2-element subsets of $H$ that realize the type
  $x_1=x_2<y_1<y_2$, i.e., 2-element subsets of columns.  If these all
  lie in the first piece of our partition, then the restriction
  $f\restr H$ is constant in each column, so it is $g\circ\pi_1$ for
  some $g:\omega\to\omega$.  Furthermore, as $\scr U=\pi_1(\scr G)$ is
  selective, by Proposition~\ref{p1-gen}, $g$ is constant or
  one-to-one on a set $A\in\scr U$.  Then the restriction of $f$ to
  $H\cap{\pi_1}^{-1}(A)$ satisfies the first or second conclusion of
  the corollary.

So we may assume from now on that the 2-element subsets of $H$
realizing $x_1=x_2<y_1<y_2$ are in the second piece of our partition.
That is, $f\restr H$ is one-to-one in each column.  We shall complete
the proof by showing that $f$ is one-to-one on all of $H$.  Suppose,
toward a contradiction, that $f$ took the same value at two elements
of $H$, necessarily in different columns.  Then the set of those two
elements realizes one of the three 2-types
\begin{align*}
  x_2&<x_1<y_1<y_2\\
x_1&<x_2<y_1<y_2\\
x_1&<y_1<x_2<y_2,
\end{align*}
because these are the only 2-types that don't have $x_1=x_2$.  By the
homogeneity of $H$, all 2-element sets $\{\sq{p_1,q_1},\sq{p_2,q_2}\}$
realizing that 2-type have $f(p_1,q_1)=f(p_2,q_2)$.  We can associate
to each of the three relevant 2-types a 3-type in which $x_2=x_3$,
namely
\begin{align*}
  x_2&=x_3<x_1<y_1<y_2<y_3\\
x_1&<x_2=x_3<y_1<y_2<y_3\\
x_1&<y_1<x_2=x_3<y_2<y_3,
\end{align*}
respectively.  This 3-type (like all types) is realized in $H$, say by
$\{\sq{p_1,q_1},\sq{p_2,q_2},\sq{p_3,q_3}\}$.  Then both of the pairs
$\{\sq{p_1,q_1},\sq{p_2,q_2}\}$ and $\{\sq{p_1,q_1},\sq{p_3,q_3}\}$
realize the 2-type that guarantees $f(p_1,q_1)=f(p_2,q_2)$ and
$f(p_1,q_1)=f(p_3,q_3)$.  Therefore $f(p_2,q_2)=f(p_3,q_3)$.  But
$p_2=p_3$, so this contradicts the fact that $f$ is one-to-one on
columns in $H$.
\end{pf}

\begin{cor}
  The generic ultrafilter \scr G is a Q-point.  The only non-principal
  ultrafilters strictly below it in the Rudin-Keisler order are the
  isomorphic copies of $\scr U=\pi_1(\scr G)$.
\end{cor}

\begin{pf}
  This follows from the preceding corollary by the same proofs as for
  Corollaries~\ref{sum-Q} and \ref{sum-RK} above.
\end{pf}

\subsection{Weak P-point}

The results proved so far about the \bbb P-generic ultrafilter \scr G
mirror the properties of selective-indexed sums of selective
ultrafilters proved in Section~\ref{sum}.  Nevertheless, there is an
important difference between \scr G and these sums; in particular,
\scr G is not such a sum.  

Notice that any sum $\scr U\sm_n\scr V_n$ is, in the Stone-\v Cech
compactification $\beta(\omega^2)$, the limit with respect to \scr U
of copies in columns of $\omega^2$ of the ultrafilters $\scr V_n$,
namely the images $i_n(\scr V_n)$ where
$i_n:\omega\to\omega^2:y\mapsto\sq{n,y}$.  In contrast, \scr G is not
such a limit; indeed, we shall show in this subsection that it is not
a limit point of any countable set of other, non-principal
ultrafilters.

We begin by recalling some standard terminology and results.  

\begin{df}
  A non-principal ultrafilter \scr W on a countable set $S$ is a
  \emph{weak P-point} if, for any countably many non-principal
  ultrafilters $\scr X_n\neq\scr W$ on $S$, there is a set $A\in\scr
  W$ such that $A\notin\scr X_n$ for all $n$.
\end{df}

In topological terms, this means that \scr W is not in the closure in
$\beta(S)$ of a countable set of other non-principal ultrafilters.

The terminology ``weak P-point'' is justified by the observation that
any P-point \scr W is also a weak P-point.  Indeed, given countably
many $\scr X_n$ as in the definition, we have for each $n$, since
$\scr W\neq\scr X_n$, some $A_n\in\scr W-\scr X_n$.  As \scr W is a
P-point, it contains a set $A$ almost included in each $A_n$, and this
$A$ is clearly not in any of the $\scr X_n$'s.
Unlike P-points, weak P-points can be proved to exist in $\ZFC$; see
\cite{kunen}.

\begin{thm}     \label{wpp}
  The generic ultrafilter \scr G is a weak P-point.
\end{thm}

\begin{pf}
  In accordance with the definition of ``weak P-point'', we shall need
  to consider non-principal ultrafilters $\scr X\neq\scr G$ on
  $\omega^2$.  It will be useful to distinguish four sorts of such
  ultrafilters \scr X:
  \begin{lsnum}
    \item ultrafilters \scr X such that $\pi_1(\scr X)$ is a principal
      ultrafilter,
\item ultrafilters \scr X such that $\pi_1(\scr X)$ is non-principal
  and distinct from $\scr U=\pi_1(\scr G)$,
\item ultrafilters \scr X such that $\pi_1(\scr X)=\scr U$ and, for
  some $f:\omega\to\omega$, the set $\{\sq{x,y}:y\leq f(x)\}$ is in
  \scr X,
\item ultrafilters $\scr X\neq\scr G$ such that $\pi_1(\scr X)=\scr U$
  and, for every $f:\omega\to\omega$, the set $\{\sq{x,y}:y>f(x)\}$ is
  in \scr X.
  \end{lsnum}
We shall show that, for any countably many ultrafilters $\scr X_n$ of
any one of these four sorts, \scr G contains a set $A$ that is in none
of these $\scr X_n$'s.  This will suffice, because then, if we are
given a countable set of $\scr X_n$'s of possibly different sorts, we
can partition it into four subsets, one for each sort, find suitable sets
$A$ for each of the four subsets, and then intersect those four $A$'s
to get a single $A$ that is in \scr G but in none of the original $\scr
X_n$'s.  In other words, it suffices to treat each sort of \scr X
separately.  This we now proceed to do, starting with the easier
cases.

Suppose we are given countably many ultrafilters $\scr X_n$ of
sort~(2).  The countably many ultrafilters $\pi_1(\scr X_n)$ on
$\omega$ are non-principal and distinct from \scr U, which is (by
Proposition~\ref{p1-gen}) selective, hence a P-point, and hence a weak
P-point.  So there is a set $B\in\scr U$ that is in none of the
ultrafilters $\pi_1(\scr X_n)$.  Then ${\pi_1}^{-1}(B)$ is in \scr G
but in none of the $\scr X_n$, as required.  This completes the proof
for sort~(2).

Suppose next that we are given countably many ultrafilters $\scr X_n$
of sort~(3).  For each $n$, fix $f_n:\omega\to\omega$ such that
$\{\sq{x,y}:y\leq f_n(x)\}\in\scr X_n$.  By diagonalization, let
$g:\omega\to\omega$ eventually majorize each of the countably many
functions $f_n$.  Thus, for each $n$, the set $\{\sq{x,y}:y\leq
f_n(x)\}$ is covered by the union of $\{\sq{x,y}:y\leq g(x)\}$ and
${\pi_1}^{-1}(F_n)$ for a finite set $F_n$.  So $\scr X_n$ must
contain $\{\sq{x,y}:y\leq g(x)\}$ or ${\pi_1}^{-1}(F_n)$.  It cannot
contain the latter, because $\pi_1(\scr X_n)$ is the non-principal
ultrafilter \scr U.  So each $\scr X_n$ must contain $\{\sq{x,y}:y\leq
g(x)\}$.  But the complement of this set is in $\scr F^{\otimes2}$ and
therefore in \scr G; it therefore serves as the required $A$.  This
completes the proof for sort~(3).

Suppose next that we are given countably many ultrafilters $\scr X_n$
of sort~(4).  In contrast to the previous cases, we shall now need to
use the fact that our ultrafilters are in the forcing extension
$V[\scr G]$.  Fix names $\dot{\scr X}_n$ for the ultrafilters $\scr
X_n$, and recall that we already fixed the canonical name $\dot{\scr
  G}$ for \scr G.  We shall complete the proof for this case by
showing that any condition $A\in\bbb P$ that forces ``The $\dot{\scr
  X}_n$ are ultrafilters of sort~(4)'' can be extended to one forcing
``$\dot{\scr G}$ contains a set that is in none of the $\dot{\scr
  X}_n$.''

Let such a condition $A$ be given.  It forces that, for each $n$,
the difference $\dot{\scr G}-\dot{\scr X}_n$ is nonempty.  Since the
forcing (in its separative form) is countably closed, we can extend
$A$ to a condition $B$ that forces, for each $n\in\omega$, a specific
set $C_n$ in the ground model to be in $\dot{\scr G}-\dot{\scr
  X_n}$. By Lemma~\ref{force-in-G}, we have that $B\subseteq C_n$
modulo $\scr F^{\otimes2}$ for each $n$.   But $B$, as an extension of
$A$, forces $\dot{\scr X}_n$ to be of sort~(4) and therefore to be a
superset of the filter  $\scr F^{\otimes2}$.  So from $B\forces
C_n\notin\dot{\scr X}_n$ it follows that $B\forces B\notin\dot{\scr
  X}_n$.  Of course $B$ forces itself to be in the generic \scr G.
Summarizing, we have an extension $B$ of $A$ forcing some set, namely
$B$ itself, to be in \scr G but in none of the $\scr X_n$.   This
completes the proof for sort~(4).

It remains to treat the case where we are given ultrafilters $\scr
X_n$ of sort~(1).  So each $\pi_1(\scr X_n)$ is a principal
ultrafilter, say generated by $\{f(n)\}$.  Note that, although the
sequence \sq{\scr X_n} and the individual ultrafilters $\scr X_n$ need
not be in the ground model, $f$, being a real, is in the ground model.

We shall complete the proof by showing that any condition $A$ forcing
``The $\dot{\scr X}_n$ are ultrafilters of sort~(1)'' can be extended
to a condition $B$ forcing ``$\dot{\scr G}$ contains a set that is in none
of the $\dot{\scr X}_n$.''  We begin with a few simplifying steps,
shrinking $A$ to normalize it in certain ways.

By extending (i.e., shrinking) the given condition $A$, we can
arrange that it is standard; recall that this means that all its
nonempty sections are infinite.  We can also arrange that it forces the
$f$ introduced above to be a specific function in the ground model.
If the set $Z=\pi_1(A)-\ran(f)$ is infinite, then
$B=A\cap{\pi_1}^{-1}(Z)$ is a condition forcing that ${\pi_1}^{-1}(Z)$
is in \scr G but in none of the $\scr X_n$, so the proof is complete
in this case.  We therefore assume that $Z$ is finite.  Removing
${\pi_1}^{-1}(Z)$ from $A$, we can arrange that
$\pi_1(A)\subseteq\ran(f)$.   From now on, we assume that all these
arrangements have been made.  

The rest of the proof will consist of an $\omega$-sequence of
successive extensions of $A$, approaching the desired condition $B$.
Each step will involve a simple construction, which we isolate in the
following lemma.

\begin{la}      \label{indep}
Any countably infinite set $S$ admits a sequence \sq{\Pi_n:n\in\omega}
of partitions $\Pi_n$ of $S$ into two pieces each, such that, whenever
pieces $C_n\in\Pi_n$ are chosen for each $n$, there is an infinite set
$P\subseteq S$ almost disjoint from all the chosen $C_n$'s.
\end{la}

\begin{pf}
  It suffices to choose the partitions to be sufficiently independent.
  For example, suppose, without loss of generality, that $S$ is the
  set of finite sequences of zeros and ones and let $\Pi_n$ partition
  the sequences according to their $n\th$ term (where sequences
  shorter than $n$ are considered to have $n\th$ term zero).  For any
  chosen pieces $C_n\in\Pi_n$, their complements (the unchosen pieces)
  have the finite intersection property, so there is an infinite $P$
  almost included in all these complements.
\end{pf}

We shall construct a sequence of standard conditions 
\[
A_0\geq A_1\geq\dots\geq A_k\geq A_{k+1}\geq\dots,
\]
starting with $A_0=A$, along with an increasing sequence
$x_0<x_1<\dots$ of natural numbers and a sequence of infinite subsets
$P_k$ of $\omega$, with the following properties.
\begin{lsnum}
  \item If $i<j$ then $\{x_i\}\times P_i\subseteq A_j$.
\item If $f(n)=x_i$ then
  $A_{i+1}\forces\{x_i\}\times(\omega-P_i)\in\dot{\scr X}_n$.
\end{lsnum}
Note that we require the $A_k$'s to be a decreasing sequence in \bbb
P, not just in the separative quotient, so they are genuine subsets of
each other, not just modulo $\scr F^{\otimes2}$.  In particular, we shall
have $\pi_1(A_k)\subseteq\ran(f)$ for all $k$.

We proceed by induction, assuming that at the beginning of stage $k$
we already have $A_i$ for $i\leq k$ and $x_i$ and $P_i$ for $i<k$,
satisfying all our requirements insofar as they involve only these
initial segments of our construction.  At the beginning of stage 0, we
have the situation already described: $A_0=A$.

In stage $k$, we must define $A_{k+1}$, $x_k$, and $P_k$ so as to
maintain our requirements.  First choose $x_k$ to be any natural number (say the first, for definiteness) larger than $x_{k-1}$ and in
$\pi_1(A_k)$.  (Ignore ``larger than $x_{k-1}$ if $k=0$.)  Of course,
such an $x_k$ exists because $A_k$ is a condition.  Let $S=A_k(x_k)$
be the corresponding section of $A_k$.  Note that $S$ is infinite
because $A_k$ is standard.  Apply Lemma~\ref{indep} to obtain sequence
of partitions $\Pi_n$ as there.

Consider those natural numbers $n$ such that $f(n)=x_k$.  For each of
these, $\scr X_n$ is forced (by $A$ and a fortiori by $A_k$) to
concentrate on $\{x_k\}\times\omega$ and therefore on a set of the form
$\{x_k\}\times Z$, where $Z$ is one of the two pieces of $\Pi_n$ or
$\omega-S$.  Furthermore, since our forcing adds no new reals, the
condition $A_k$ has an extension $B$ that decides these options for each
relevant $n$.  If $B$ forces $\scr X_n$ to concentrate on
$\{x_k\}\times Z$ where $Z\in\Pi_n$, let $C_n$ be that $Z$.  For all other
$n$'s (those for which $f(n)\neq x_k$ and those for which $B$ forces
$\scr X_n$ to concentrate on $\{x_k\}\times(\omega-S)$), choose
$C_n\in\Pi_n$ arbitrarily.  By Lemma~\ref{indep}, we can find an
infinite subset of $S$ almost disjoint from all the chosen $C_n$'s;
let $P_k$ be such a set.  Note that, for each $n$ with $f(n)=x_k$, the
condition $B$ forces $\scr X_n$ to concentrate on
$\{x_k\}\times(\omega-P_k)$.  

Obtain $A_{k+1}$ from $B$ by standardizing (i.e., removing all finite
sections) and adjoining all the sets $\{x_i\}\times P_i$ for all
$i\leq k$.  None of this affects $B$ in the separative quotient, since
the union of the finite columns removed and the finitely many columns
added is in the ideal dual to $\scr F^{\otimes 2}$.  Furthermore,
$A_{k+1}\subseteq A_k$.  Indeed, we had $B\subseteq A_k$, and
standardization only shrinks $B$.  As for the additional columns
$\{x_i\}\times P_i$, the ones for $i<k$ were subsets of $A_k$ by
induction hypothesis, and the one for $i=k$ is also included in $A_k$
because $P_k\subseteq S=A_k(x_k)$.

This completes the inductive construction of the sequences of $A_k$'s,
$x_k$'s, and $P_k$'s.  Finally, let 
\[
B=\{\sq{x_k,y}:k\in\omega\text{ and }y\in P_k\}.
\]
This $B$ is a condition, as there are infinitely many $x_k$'s and each
$P_k$ is infinite.  It is an extension of all the $A_k$'s because each
of its elements \sq{x_k,y} is in $A_j$ for all $j>k$ by our inductive
construction, and therefore also for $j\leq k$ because the $A_j$'s
form a decreasing sequence.  We claim that $B$ forces all the $\scr
X_n$ to concentrate on $\omega^2-B$.  Since it forces \scr G to
concentrate on $B$, this will complete the proof of the theorem.

So consider an arbitrary $n\in\omega$.  If $f(n)$ is not of the form
$x_k$, then $\scr X_n$ concentrates on $\{f(n)\}\times\omega$, which
is disjoint from $B$, as required for the claim.  So suppose that
$f(n)=x_k$.  Then $A_{k+1}$ was constructed to force $\scr X_n$ to
concentrate on $\{x_k\}\times(\omega-P_k)$, which is disjoint from
$B$.  As an extension of $A_{k+1}$,  $B$ forces the same, and the
proof is therefore complete.
\end{pf}

\section{$\mathcal{G}$ does not have maximal Tukey
  type} \label{sec:gen2}

We prove a canonization theorem showing that every monotone map with
domain $\mathcal{G}$ is almost continuous (represented by a finitary
function) on a cofinal subset of $\mathcal{G}$.  It follows that every
ultrafilter Tukey reducible to $\mathcal{G}$ has Tukey type of
cardinality continuum.  In particular, $\mathcal{G}$ is strictly below
the top Tukey degree, thus answering a question of Blass stated during
a talk at the Fields Institute, September 2012.  Moreover, we show
that $\mathcal{G}\not\ge_T [\om_1]^{<\om}$, answering a question of
Raghavan in \cite{RaghavanArxivOct2012}.

The results in this section were obtained by Dobrinen and completed on
the following dates.  Theorems \ref{thm.canon} and
\ref{thm.answerBlass} were completed on October 8, 2012.  Theorem
\ref{thm.notaboveom_1} was completed on October 31, 2012.  Proposition
\ref{prop.1<T2} was completed on November 1, 2012.

We begin by recalling some useful facts.  
For ultrafilters  $\mathcal{U},\mathcal{V}$,
a map $f:\mathcal{U}\ra\mathcal{V}$ is called {\em monotone} if whenever $u\contains u'$ are in $\mathcal{U}$, then $f(u)\contains f(u')$.
By [Fact 6, in \cite{dt}], whenever $\mathcal{U}\ge_T\mathcal{V}$,
 there is a monotone convergent map $f:\mathcal{U}\ra\mathcal{V}$ witnessing this reduction.
Recall the following  [Theorem 20, in \cite{dt}] canonizing Tukey reductions from P-points as continuous maps.
Here, $\mathcal{P}(\om)$ is endowed with 
 the Cantor topology on $2^{\om}$ by associating subsets of $\om$ with their characteristic functions.

\begin{thm}[Dobrinen/Todorcevic, \cite{dt}]\label{thm.DTcts}
Suppose $\mathcal{U}$ is a P-point on $\om$ and $\mathcal{V}$ is an arbitrary ultrafilter on a countable base set such that $\mathcal{U}\ge_T \mathcal{V}$.
For each monotone convergent map  $f:\mathcal{U}\ra\mathcal{V}$,
 there is an $\tilde{x}\in\mathcal{U}$
such that $f\re (\mathcal{U}\re\tilde{x})$ is continuous.
Moreover, there is 
  a continuous monotone map
$f^*:\mathcal{P}(\om)\ra\mathcal{P}(\om)$ 
such that $f^*\re(\mathcal{U}\re\tilde{x})=f\re(\mathcal{U}\re\tilde{x})$.
Hence, there is a continuous monotone convergent map $f^*\re\mathcal{U}$  from $\mathcal{U}$ into $\mathcal{V}$ which extends $f\re(\mathcal{U}\re\tilde{x})$.
\end{thm}

The proof of Theorem \ref{thm.DTcts} holds whenever $\mathcal{V}$ is an ultrafilter on any countable base set $B$, the topology given by enumerating $B$ in order type $\om$ and considering the Cantor topology on $2^B$.  
In the next proposition, we shall apply this theorem with $\mathcal{G}$ in place of $\mathcal{V}$.

Recall Proposition \ref{p1-gen}
which implies
that  $\pi_1(\mathcal{G})$ is selective (hence a P-point) and that $\pi_1(\mathcal{G})<_{RK}\mathcal{G}$.
It follows that $\pi_1(\mathcal{G})\le_T\mathcal{G}$.
We begin this section by showing the stronger fact  that the inequality is  also strict for the Tukey reduction.

\begin{prop}\label{prop.1<T2}
 $\mathcal{G}>_T \pi_1(\mathcal{G})$.
\end{prop}

\begin{proof}
We shall show that for
each monotone  map $f:\mathcal{P}(\om)\ra\mathcal{P}(\om^2)$ in $V$,
$f\re\pi_1(\mathcal{G})$ is not a convergent map from $\pi_1(\mathcal{G})$ into $\mathcal{G}$ in $V[\mathcal{G}]$.
This shows, in particular,
 there are no monotone continuous maps in $V$ witnessing a Tukey reduction from $\pi_1(\mathcal{G})$ into $\mathcal{G}$.
Since $\pi_1(\mathcal{G})$ is a P-point,
 by Theorem \ref{thm.DTcts},  every Tukey reduction from $\pi_1(\mathcal{G})$ to any other ultrafilter is witnessed by a  monotone continuous  map.
Since  $\bP$ is $\sigma$-closed, 
 every continuous map  $f:\mathcal{P}(\om)\ra\mathcal{P}(\om^2)$ in $V[\mathcal{G}]$ is actually in $V$.
Thus, $\pi_1(\mathcal{G})\not\ge_T\mathcal{G}$.

Let $x\in\bP$ be given and let $f:\mathcal{P}(\om)\ra\mathcal{P}(\om^2)$ be a monotone  map in $V$.
If there is a $y\in \bP$ such that $y\sse x$  and $f(\pi_1[y])\cap y\not\in \bP$,
then $y\forces f(\pi_1(\dot{\mathcal{G}}))\not\sse\dot{\mathcal{G}}$.
Otherwise,  for all $y\in\bP$ with $y\sse x$, $f(\pi_1[y])\cap y \in \bP$.
If there is a $y\in\bP$ such that $y\sse x$ and  for all $z\in\bP$ with $z\sse y$, $f(\pi_1[z])\not\sse y$,
then $y\forces$ ``$f\re\pi_1(\dot{\mathcal{G}}):\pi_1(\dot{\mathcal{G}})\ra\dot{\mathcal{G}}$ is not a convergent map."

If none of the above cases holds, then for each $y\in \bP$ with $y\sse x$, (a) $f(\pi_1[y])\cap y \in \bP$,  and (b) there is a $z\in\bP$ with $z\sse y$ such that $f(\pi_1[z])\sse y$
Fix some $y,w\in\bP$ with   $y,w\sse x$ such that 
$\pi_1[y]=\pi_1[w]=\pi_1[x]$ but
$y\cap w=\emptyse$.
Take a $y'\in\bP$ with $y'\sse y$ such that  $f(\pi_1[y'])\sse y$.
Next, take a $w'\in\bP$ with $w'\sse w$ such that $\pi_1[w']\sse \pi_1[y']$;  then take a $w''\in\bP$ with $w''\sse w'$ such that $f(\pi_1[w''])\sse w'$.
Then $\pi_1[w'']\sse \pi_1[y']$, so $f(\pi_1[w''])\sse f(\pi_1[y'])$, since $f$ is monotone.  Hence, $f(\pi_1[w''])\sse y$. 
At the same time,
 $f(\pi_1[w''])\sse w' \sse w$. 
Thus,  $f(\pi_1[w''])\sse  w\cap y=\emptyse$.
Hence, it is dense to force that $\dot{f}$ is constantly zero on some cofinal subset of $\pi_1(\dot{\mathcal{G}})$.

Therefore, for each monotone  map $f:\mathcal{P}(\om)\ra\mathcal{P}(\om^2)$ in $V$,
it is dense to force that ``$f\re\pi_1(\dot{\mathcal{G}})$ is not a convergent map from $\pi_1(\dot{\mathcal{G}})$ into $\dot{\mathcal{G}}$."
\end{proof}

\begin{notat}\label{notation.standard}
Throughout this section, we shall let $\bP_s$ denote the collection of  the standard forcing conditions; that is, those conditions such that each non-empty fiber is infinite.
We shall let $\mathcal{G}_s$ denote the members of $\mathcal{G}$ which are standard; that is, $\mathcal{G}_s=\mathcal{G}\cap\bP$.
For $u\in \bP_s$ and $n<\om$, we shall let $u\cap n^2$ denote $u\cap (n\times n)$.
Recall that for any $p\in\bP$ and $i\in\om$, $p(i)$ denotes $\{j\in\om:(i,j)\in p\}$, the $i$-th fiber of $p$.
\end{notat}

The next theorem provides a canonization for all monotone maps with domain $\mathcal{G}_s$  in terms of monotone finitary  maps on the   base $[\om^2]^{<\om}$.
The argument of Theorem \ref{thm.canon} combines some  of the key  traits of the proof of Theorem 20 in \cite{dt} for canonizing monotone convergent maps on P-points and the proof of Theorem 13
in \cite{Dobrinen11} canonizing monotone convergent maps on iterated Fubini products of P-points, making new adjustments for this setting.

\begin{defn}\label{defn.represents}
We say that a monotone map $f$ on $\mathcal{G}_s\re x$ is {\em represented by a  monotone finitary map $\vp$} if
there is a map $\vp:[\om^2]^{<\om}\ra [\om]^{<\om}$ such that 
for all $s,t\in [\om^2]^{<\om}$,
\begin{enumerate} 
\item
(Monotonicity)
$s\sse t\ra \vp(s)\sse\vp(t)$;
\item
($\vp$ represents $f$)
 For each
$u\in\mathcal{G}_s\re x$,
$f(u)=\bigcup_{n<\om}\vp(u\cap n^2)$.
\end{enumerate}
\end{defn}

\begin{thm}[Canonization of monotone maps as almost continuous]\label{thm.canon}
In $V[\mathcal{G}]$,
for each monotone function $f:\mathcal{G}\ra\mathcal{P}(\om)$,
there is an $x\in\mathcal{G}_s$ such that $f\re(\mathcal{G}_s\re x)$ is represented by a monotone finitary  map. 
Hence, every monotone cofinal map $f:\mathcal{G}\ra\mathcal{V}$ is represented by a monotone finitary map on the filter base consisting of members of $\mathcal{G}_s$ below some $x\in\mathcal{G}_s$, where $\mathcal{V}$ is any ultrafilter on a countable base.
\end{thm}

\begin{proof}
Let $\dot{f}$ be a $\bP$-name 
such that $\forces$ ``$\dot{f}:\dot{\mathcal{G}}\ra\mathcal{P}(\om)$ is
  monotone."
Unless stated otherwise, all conditions are assumed to be in $\bP_s$.
Fix a $p\in \bP_s$, and let $p_{-1}=p$.
In the first step of the proof, we shall construct a sequence $p\ge x_0\ge p_0\ge x_1\ge p_1\ge \dots$ of members of $\bP_s$ such that 
 the following (a)-(c)  hold for each $k<\om$:

\begin{enumerate}
\item[(a)]
 $x_{k+1}\sse x_k$.
\item[(b)]
There is a sequence $i_0<i_1<\dots$ such that
$i_0=\min(\pi_1[p])$,
and in general,
 $$i_{k+1}=\min(\pi_1[x_k]\setminus\{i_0,\dots, i_k\})=\min(\pi_1[x_k]\setminus (i_k+1));$$  
and  for each $n<\om$,
$$\pi_1[x_{k+n}]\cap(i_{k+1}+1)=\pi_1[x_k]\cap(i_{k+1}+1)=\{i_0,\dots,i_{k+1}\}.$$
\item[(c)]
Let $S_k$ denote the set of all triples $(s,t,j)\in\mathcal{P}(\{i_0,\dots,i_k\}\times(i_k+1))
\times\mathcal{P}(\{i_0,\dots,i_k\})\times(i_k+1)$ such that $\pi_1[s]\sse t$.
Fix an enumeration of $S_k$ as $(s_k^l,t_k^l,j_k^l)$, $l\le l_k:=|S_k|-1$.
For all $k<\om$ and $l\le l_k$, $x_k$ and $p_k$ satisfy the following:

{\bf If} there are  $v,q\in\bP_s$  with $q\le v\le  p_{k-1}$ such that 
\begin{enumerate}
\item[(i)]
$v\cap(i_k+1)^2=s_k^l$;
\item[(ii)]
$\pi_1[v]\cap(i_k+1)=t_k^l$;
\item[(iii)]
For all $i\in t_k^l$, $v(i)\setminus(i_k+1)\sse x_k(i)$;
\item[(iv)]
$v\cap((i_k,\om)\times\om)\sse x_k$;
\item[(v)]
$q\forces j_k^l\not\in\dot{f}(v)$;
\end{enumerate}
{\bf Then} 
for each $u\in\bP_s$ such that 
\begin{enumerate}
\item[(vi)]
$u\cap(i_k+1)^2=s_k^l$;
\item[(vii)]
$\pi_1(u)\cap(i_k+1)=t_k^l$;
\item[(viii)]
For all $i\in t_k^l$, $u(i)\setminus(i_k+1)\sse x_k(i)$;
\item[(ix)]
$u\cap((i_k,\om)\times\om)\sse x_k$;
\end{enumerate}
we have that 
 $p_k\forces j_k^l\not\in\dot{f}(u)$.
\end{enumerate}
Note that it follows from (b) that $\pi_1[x_k]\ra\{i_0,i_1,\dots\}$ as $k\ra\infty$.
\vskip.1in

We now begin the recursive construction of the sequences $(x_n)_{n<\om}$ and $(p_n)_{n<\om}$.
Let $p_{-1}=x_{-1}=p$.
Let $k\in\om$ be given, and suppose   we have chosen $x_{k-1},p_{k-1}$.
If $k=0$, let $i_0=\min(\pi_1[p])$;
if $k\ge 1$, let $i_k=\min(\pi_1[x_{k-1}]\setminus (i_{k-1}+1))$.
Recall that  $S_k$ denotes
 the set of all triples $(s,t,j)$ such that
$s\in\mathcal{P}(\{i_0,\dots,i_k\}\times(i_k+1))$,
$t\in\mathcal{P}(\{i_0,\dots,i_k\})$,
and $j\in (i_k+1)$ such that $\pi_1[s]\sse t$.
Let $x_k^{-1}=x_{k-1}$ and $p_k^{-1}=p_{k-1}$.
Suppose we have chosen $x_{k}^{l-1}$ and $p^{l-1}_{k}$ for $ l\le  l_{k}=|S_k|-1$.

{\bf  If}
there are $v,q\in\bP_s$ and $q\le v\le  p_{k}^{l-1}$ such that
\begin{enumerate}
\item[(i)]
$v\cap(i_k+1)^2=s^{l}_k$;
\item[(ii)]
$\pi_1[v]\cap(i_k+1)=t_k^l$;
\item[(iii)]
For all $i\in t_{k}^{l}$, $v(i)\setminus (i_{k}+1)\sse  x_{k}^{l-1}(i)$;
\item[(iv)]
$v\cap ((i_{k},\om)\times\om)\sse x_k^{l-1}$;
\item[(v)]
$q\forces j_{k}^{l}\not\in\dot{f}(v)$;
\end{enumerate}
{\bf then} 
take  $p_{k}^{l}$ and $v_{k}^{l}$ be some such $q$ and $v$.
Note that $p^l_k\forces j^l_k\not\in\dot{f}(v^l_k)$.
Hence, by monotonicity, $p^l_k\forces j^l_k\not\in \dot{f}(v)$, for every $v\sse v^l_k$.
In this case, let 
\begin{align}
x_{k}^{l}=
& \bigcup \{\{i\}\times (v_{k}^l(i)\setminus(i_k+1)) : i\in t_{k}^l\}\cr &\cup\bigcup\{\{i\}\times x_k^{l-1}(i)   : i\in\{i_0,\dots,i_{k}\}\setminus t_{k}^l\}\cr
&\cup
v_{k}^l\cap ((i_{k},\om)\times\om).
\end{align}
Thus, $x^l_k$ is empty on the square $(i_k+1)^2$;
on indices $i\in t^l_k$, the $i$-th fiber of $x^l_k$ equals the $i$-th fiber of $v^l_k$ above $i_k$;
on indices $i\le i_k$ which are not in $t^l_k$, the $i$-th fiber of $x^l_k$ equals the $i$-th fiber of $x^{l-1}_k$;
and for all $i>i_k$, the $i$-th fiber of $x^l_k$ is exactly the same as $v^l_k$.

{\bf Otherwise}, 
for all $v\le p^{l-1}_k$ satisfying (i) - (iv), 
there is no $q\le v$ which forces $j^l_k$ to not be in $\dot{f}(v)$.
Thus, for all $v\le p^{l-1}_k$ satisfying (i) - (iv), $v\forces j^l_k\in\dot{f}(v)$.
In particular, $p^l_k\forces j^l_k\in\dot{f}(v^l_k)$.
In this case,
let $p^{l}_{k}=p^{l-1}_{k}$ and  $x_{k}^{l}=x^{l-1}_{k}$,
and define 
\begin{equation}
v_{k}^l= s_k^l\cup (x_k^l\cap((t^l_k\cup (i_k,\om))\times\om)).
\end{equation}

In the `If' case, $p^l_k\forces j^l_k\not\in \dot{f}(v^l_k)$.
In the `Otherwise' case, $p^l_k\forces j^l_k\in\dot{f}(v^l_k)$.
Thus,
\begin{enumerate}
\item[$(*)$]
For all $k<\om$ and $l\le l_k$,
$p^l_k$ decides the statement ``$j^l_k\in\dot{f}(v^l_k)$".
\end{enumerate}

After the $l_k$ steps, let $p_k=p_k^{l_k}$ and $x_{k}=x_{k}^{l_k}$.
Note that $x_k\cap(i_k+1)^2=\emptyset$, and $x_{k}\sse x_{k-1}$ by (iii) and (iv).
This ends the recursive  construction of the $p_k$ and $x_k$.
It is not hard to see that the sequence of $x_k,p_k$ satisfies (a) - (c). 
\vskip.1in

In the second step, we diagonalize through the $x_n$ to obtain $x$ as follows.
Let $k_0=0$.
Take an $a_{0,0}$ in the fiber  $x_0(i_0)$.
Choose $k_1$ so that $i_{k_1}> a_{0,0}$.
In general, given $k_n$,
for each $j\le i_{k_n}$,
take an $a_{j,n}\in x_{k_n}(i_j)\setminus(i_{k_n}+1)$.
Then choose $k_{n+1}$ so that $i_{k_1}>\max\{a_{j,n}: j\le k_n\}$.
Let 
$$x=\{(i_j,a_{j,n}): n<\om, j\le k_n\}.
$$
Note that $x\in\bP$, $x\sse p$, $\pi_1[x]=\{i_{k_0},i_{k_1},i_{k_2},\dots\}$, and in fact, $x\sse^* x_k$ for all $k<\om$.
Thus, also $x\le p_k$, for all $k<\om$.
Moreover, 
for each $n$, $x\setminus (i_{k_n}+1)^2\sse x_{k_n}$.
From now on,  we work below $x$.

Let us establish some useful notation.
Given any $(s,t)\in[\om^2]^{<\om}\times[\om]^{<\om}$, $k\ge\max(\pi_1[s]\cup t)$, and $v\sse x$,
define
\begin{equation}
z(k,s,t,v)= s\cup(\bigcup_{i\in t}\{i\}\times( v(i)\setminus (k+1)))
\cup( v\cap((k,\om)\times\om)),
\end{equation}
and
\begin{equation}
\dot{f}(k,s,t,v)
=\dot{f}(z(k,s,t,v)).
\end{equation}
Thus, $z(k,s,t,v)$ is the  following condition $z\in\bP_s$:  Inside the square $(k+1)\times(k+1)$, $z$ is exactly $s$.
For indices $i\le k$, $z$ has nonempty fibers if and only if $i\in t$; and if $i\in t$, then the tail of the $i$-th fiber of $z$ above $k$ is exactly the tail of the  $i$-th fiber of $v$ above $k$.
All fibers of $z$ with index greater than $k$ are exactly the fibers of $v$ with index greater than $k$.

Our construction was geared toward establishing  the following.

\begin{claim}\label{claim.dagger}
Let $v\sse x$,
 $n<\om$, 
 and $j\le i_{k_n}$ 
be given; and
let 
$l\le l_{k_n}$ be the integer satisfying 
$s_{k_n}^l=v\cap (i_{k_n}+1)^2$, $t_{k_n}^l=\pi_1[v]\cap(i_{k_n}+1)$,
and $j^l_{k_n}=j$.
Then 
$$ v\forces  j\in \dot{f}(v)  
\Llra 
x\forces j\in\dot{f}(i_{k_n},s_{k_n}^l,t_{k_n}^l, x)
\Llra
x\forces  j\in\dot{f}(v_{k_n}^l).$$
\end{claim}

\begin{proof} Assume 
 $v\sse x$,
 $n<\om$, 
 $j\le i_{k_n}$, and  $l\le l_{k_n}$ satisfy the hypothesis.
By $(*)$,
$p^l_{k_n}$ decides whether or not $j\in\dot{f}(v^l_{k_n})$.
Since $x\le p^l_{k_n}$,
$x$ also decides whether or not $j\in\dot{f}(v^l_{k_n})$.

Suppose $x\forces j\not\in\dot{f}(v_{k_n}^l)$.
Our  choice of $l$ implies that 
\begin{align}\label{eq.sse}
v&=z(i_{k_n},s_{k_n}^l,t^l_{k_n},v)\sse
z(i_{k_n},s_{k_n}^l,t^l_{k_n},x)\sse
z(i_{k_n},s_{k_n}^l,t^l_{k_n},x_{k_n})\\
&\sse
z(i_{k_n},s_{k_n}^l,t^l_{k_n},x_{k_n}^l)\sse
z(i_{k_n},s_{k_n}^l,t^l_{k_n},v^l_{k_n})=v^l_{k_n}.
\end{align}
By monotonicity of $\dot{f}$,  
we have
$x\forces j\not\in\dot{f}(i_{k_n},s_{k_n}^l,t^l_{k_n},x)$, and 
$x\forces j\not\in \dot{f}(v)$.
Hence, also $v\forces j\not\in \dot{f}(v)$.

Suppose now that $x\forces j\in\dot{f}(v_{k_n}^l)$.
Then in the construction of $x_{k_n}^l$, we were in the `Otherwise' case.
Thus, for all pairs $q\le v'\le p_{k_n}^{l-1}$ 
satisying (i) - (iv) for the pair $k_n,l$,
we have that
$q\forces j\in \dot{f}(v')$.
In particular,
$v\forces j\in\dot{f}(v)$.
Further, since $x\le z(i_{k_n},s_{k_n}^l,t_{k_n}^l,x)$,
we have
$x\forces j\in\dot{f}(i_{k_n},s_{k_n}^l,t_{k_n}^l,x)$.
\end{proof}

It follows immediately from Claim \ref{claim.dagger} that 
\begin{enumerate}
\item[$(\dag)$]
For each  $y\sse x$ and $j<\om$, taking $n$ so that $j<i_{k_n}$, the pair of finite sets $(y\cap (i_{k_n}+1)^2,
\pi_1[y]\cap(i_{k_n}+1))$  completely determines whether or not 
$y\forces j\in\dot{f}(y)$.
\end{enumerate}

Now we define a finitary monotone function $\psi$ which $x$ forces to represent $\dot{f}$ on a cofinal subset of $\dot{\mathcal{G}}$.
For all pairs $(s,t)\in[x]^{<\om}\times[\pi_1[x]]^{<\om}$, take $n$ least such that 
$i_{k_n}\ge \max(\pi_1[s]\cup t)$ and 
define 
\begin{equation}\label{def.6}
\psi(s,t)=\{j\le i_{k_n}: x\forces j\in
\dot{f}(i_{k_n},s,t,x)\}.
\end{equation}
Note that $\psi$ is monotone:
Suppose $s\sse s'$ and $t\sse t'$.
Let $n$ be least such that $\max(\pi_1[s]\cup t)\le i_{k_n}$, and 
let $n'$ be least such that $\max(\pi_1[s']\cup t')\le i_{k_{n'}}$.
Let $s''=s'\cap (i_{k_n}+1)^2$ and $t''=t'\cap(i_{k_n}+1)$.
Let $j\le i_{k_n}$.
By  Claim \ref{claim.dagger}, (\ref{def.6}) and monotonicity of $\dot{f}$, 
\begin{align}
j\in\psi(s,t)& \Llra  x\forces j\in \dot{f}(k_n,s,t,x)\cr
& \Lra 
x\forces j\in \dot{f}(k_n,s'',t'',x)\cr
& \Lra
x\forces j\in \dot{f}(k_{n'},s',t',x)\cr
&\Llra
j\in\psi(s',t').
\end{align}

\begin{claim}\label{claim.psigivesf}
If $x$ is in $\mathcal{G}$, then
$\psi$ represents $f$ on $\mathcal{G}_s\re x$. 
That is, for each  $y\in\bP_s$ with $y\sse x$,
$y\forces \dot{f}(y)=\bigcup\{\psi(y\cap(i_{k_n}+1)^2,\pi[y]\cap (i_{k_n}+1)):n<\om\}$.
\end{claim}

\begin{proof}
Let $y\in\bP_s$ such that $y\sse x$.
By the definition of $\psi$,
for each $j$, if $x\forces j\in\dot{f}(y)$, 
then $j\in \psi(y\cap(i_{k_n}+1)^2,\pi_1[y]\cap (i_{k_n}+1))$ for every  $i_{k_n}\ge j$.
On the other hand, if $j\in \psi(y\cap(i_{k_n}+1)^2,\pi_1[y]\cap (i_{k_n}+1))$,
then 
$x\forces j\in
\dot{f}(k_n, y\cap(i_{k_n}+1)^2, \pi_1[y]\cap (i_{k_n}+1),x)$.
Thus,
$y\forces 
 j\in
\dot{f}( y)$, by Claim \ref{claim.dagger}.
\end{proof}

We  simplify  $\vp$ a bit more.
Define the  map $\vp:[x]^{<\om}\ra[\om]^{<\om}$ by letting
 $\vp(s)=\psi(s,\pi_1[s])$,
for each $s\in[x]^{<\om}$.

\begin{claim}
$\vp$ is monotone and represents $\dot{f}\re\bP_s\re x$; that is, for each standard $y\sse x$, 
\begin{equation}
y\forces\dot{f}(y)=\bigcup\{\vp(y\cap(i_{k_n}+1)^2):n<\om\}.
\end{equation}
\end{claim}

\begin{proof}
Suppose $s\sse s'$.
Then $\pi_1[s]\sse \pi_1[s']$, so by monotonicity of $\psi$, we also have $\vp(s)\sse\vp(s')$.

Given $y\in\bP_s$ with $y\sse x$, 
$y\forces\dot{f}(y)\contains\bigcup\{\vp(y\cap(i_{k_n}+1)^2):n<\om\}$, 
by definition of $\vp$ and Claim \ref{claim.psigivesf}.
On the other hand, suppose $y\forces j\in\dot{f}(y)$.
Taking $n$ least such that $i_{k_n}\ge j$, Claim \ref{claim.psigivesf}  and $(\dag)$ imply $j\in\psi(y\cap(i_{k_n}+1)^2,\pi_1[y]\cap(i_{k_n}+1))$. 
Take  $m$ least such that
$\pi_1[y\cap(i_{k_m}+1)^2]\contains \pi_1[y]\cap(i_{k_n}+1)^2$.
Let $s$ be $y\cap(i_{k_m}+1)^2$ restricted to indices in $\pi_1[y]\cap(i_{k_n}+1)^2$.
That is, let $a\in\om^2$ be in $s$ if and only if $a\in y\cap(i_{k_m}+1)^2$ and $\pi_1[a]\in \pi_1[y]\cap(i_{k_n}+1)^2$.
Monotonicity of $\psi$ implies  that 
$j\in\psi(s,\pi_1[y]\cap(i_{k_n}+1))$.
Since $\pi_1[s]=\pi_1[y]\cap(i_{k_n}+1)$,
we have 
$j\in\vp(s)$.
Since $\vp$ is monotone and $y\cap(i_{k_m}+1)^2\contains s$, 
it follows that $j\in\vp(y\cap(i_{k_m}+1)^2)$.
\end{proof}

It follows that in $V[\mathcal{G}]$, for every monotone  map $f:\mathcal{G}\ra\mathcal{P}(\om)$, there is an $x\in\mathcal{G}_s$ such that 
$f\re(\mathcal{G}_s\re x)$ is represented by a monotone 
finitary map $\vp:[x]^{<\om}\ra[\om]^{<\om}$.

If $\mathcal{V}$ is an ultrafilter on a countable base set, we may without loss of generality assume that base is $\om$.
Thus, each monotone cofinal map $f:\mathcal{G}\ra\mathcal{V}$ is represented on a base of the form $\mathcal{G}_s\re x$ by a finitary map $\vp$ on $[x]^{<\om}$, for some $x\in\mathcal{G}$.
\end{proof}

\begin{rmk}
We point out several key observations.
First, the map $\vp$ in Theorem \ref{thm.canon}
actually generates a monotone map from $f^*:\mathcal{P}(\om^2)\ra\mathcal{P}(\om)$, by defining
$f^*(y)=\bigcup_{n<\om}\vp(y\cap (i_{k_n}+1)^2)$, for each $y\sse\om^2$.
It is easy to  check that  $f^*\re\mathcal{G}_s\re x=f\re\mathcal{G}_s\re x$.
Thus,  if $f$  is a monotone cofinal map from $\mathcal{G}$ into some ultrafilter $\mathcal{V}$,  then
$f^*\re\mathcal{G}$ is also a monotone cofinal map from $\mathcal{G}$ into $\mathcal{V}$ which is moreover finitely represented.

Second, the map $f\re\mathcal{G}_s\re x$ is {\em almost continuous} in the following senses.
First, being finitely represented, it quite similar to continuous maps on $\mathcal{P}(\om^2)$, given the Cantor topology on $\mathcal{P}(\om^2)$.
Of more interest, though, is the property $(\dag)$, that for each $y\in\mathcal{G}_s\re x$ and $n<\om$,
the map $\psi$ can decide by time $i_{k_n}$ whether or not $n$ is in $f(y)$.
$\psi$ can do this because it not only looks at $y\cap i_{k_n}$ but also sees the future  of whether or not the fibers $y(l)$ will be nonempty, for $l\le i_{k_n}$.
\end{rmk}

We now answer a question of Blass,  by applying Theorem \ref{thm.canon} to  show that $\mathcal{G}$ does not have maximal Tukey type.

\begin{thm}\label{thm.answerBlass}
For every $\mathcal{V}\le_T\mathcal{G}$, the Tukey type of $\mathcal{V}$ has cardinality continuum.
In particular, the Tukey type of $\mathcal{G}$ has cardinality continuum; hence, $\mathcal{G}$ does not have maximal Tukey type.
\end{thm}

\begin{proof}
By Theorem \ref{thm.canon}, every Tukey reduction $\mathcal{V}\le_T\mathcal{G}$ is witnessed by a function $f$ represented by a finitary monotone  map 
$\vp:[x]^{<\om}\ra[\om]^{<\om}$ for some $x\in \mathcal{G}$.
Since there are only continuum such maps, it follows that
for each $\mathcal{V}\le_T\mathcal{G}$, the Tukey type of $\mathcal{V}$ has cardinality $\mathfrak{c}$.
In particular, $\mathcal{G}$ does not have maximal Tukey type.
\end{proof}

In Theorem \ref{thm.notaboveom_1} we prove that  $([\om_1]^{<\om},\sse)\not\le_T(\mathcal{G}\contains)$.
This answers a question
 of Raghavan in \cite{RaghavanArxivOct2012},
where he
stated that it is easy to see that if in $V$, $\mathfrak{h}(\mathcal{P}(\om^2)/(\fin\times\fin))>\om_1$,
then in $V[\mathcal{G}]$, $([\om_1]^{<\om},\sse)\not\le_T(\mathcal{G}\contains)$, but that 
he did not know whether this holds in general.
This also gives a second  proof that $\mathcal{G}$ has Tukey type strictly below the maximal type.

We first prove the following fact, which will be used in the proof of Theorem \ref{thm.notaboveom_1}.

\begin{prop}\label{fact.monoom_1}
For any ultrafilter $\mathcal{U}$,
if  $(\mathcal{U},\contains)\ge_T([\om_1]^{<\om},\sse)$,
then there is a monotone convergent map $f:\mathcal{U}\ra [\om_1]^{<\om}$ witnessing this.
\end{prop}

\begin{proof}
Suppose $\mathcal{U}\ge_T[\om_1]^{<\om}$.
Then there is a Tukey map $g:[\om_1]^{<\om}\ra\mathcal{U}$ mapping unbounded subsets of $([\om_1]^{<\om},\sse)$ to unbounded subsets of  $(\mathcal{U},\contains)$.
For each $u\in\mathcal{U}$, define $f(u)=\bigcup\{s\in[\om_1]^{<\om}:g(s)\contains u\}$.
Note that $\{s\in[\om_1]^{<\om}:g(s)\contains u\}=g^{-1}(\{w\in\mathcal{U}:w\contains u\})$ which is  the $g$-preimage of a bounded subset of $\mathcal{U}$.
 Since $g$ is a Tukey map,
 $\{s\in[\om_1]^{<\om}:g(s)\contains u\}$ is bounded by some member of $[\om_1]^{<\om}$;
so $f(u)$ is well-defined.

If $u\contains v$ are in $\mathcal{U}$,
then  
$g^{-1}\{w\in\mathcal{U}:w\contains u\}\sse g^{-1}\{w\in\mathcal{U}:w\contains v\}$.
Therefore,
 $\bigcup g^{-1}\{w\in\mathcal{U}:w\contains u\}\sse\bigcup g^{-1}\{w\in\mathcal{U}:w\contains v\}$;
so $f(u)\sse f(v)$.
Hence, $f$ is monotone.

To see that $f$ is a convergent map, it suffices to check that the $f$-image of $\mathcal{U}$ is cofinal in $[\om_1]^{<\om}$, since $f$ is monotone.
Let $s\in[\om_1]^{<\om}$.
Then $s\in \{t\in[\om_1]^{<\om}:g(t)\contains g(s)\}$,
so
$s\sse\bigcup\{t\in[\om_1]^{<\om}:g(t)\contains g(s)\}= f(g(s))$, which is a member of $f[\mathcal{U}]$.
\end{proof}

\begin{thm}\label{thm.notaboveom_1}
 $(\mathcal{G},\contains)\not\ge_T ([\om_1]^{<\om},\sse)$.
\end{thm}

\begin{proof}
Let $p\in\bP$ and $\dot{f}$ be a $\bP$-name such that  $p\forces$  ``$\dot{f}$ is a monotone map from $\dot{\mathcal{G}}$ into $[\om_1]^{<\om}$."
We shall show that there is an $x\le  p$ such that $x\forces$ ``$\dot{f}[\dot{\mathcal{G}}]$ is not cofinal in $[\om_1]^{<\om}$".
By Fact \ref{fact.monoom_1}
it will follow that, in $V[\mathcal{G}]$, $\mathcal{G}\not\ge_T[\om_1]^{<\om}$.

\it Case 1. \rm Suppose that  for all $p'\le p$, $p'\forces\dot{f}(p')=0$.
Let $a\in\bP$ with $a\sse p$ be given, and take any $r\in\bP$ with $r\sse a$.
Since $p\forces\dot{f}$ is monotone, it follows that $r\forces \dot{f}(r)\contains\dot{f}(a)$.
 $r\le p$ implies $r\forces \dot{f}(r)=0$.
Thus, $r\forces \dot{f}(a)=0$.
Since the collection of all  $r\in\bP$ such that $r\sse a$ is dense below $a$,
it follows that $a\forces \dot{f}(a)=0$.
Thus, $p\forces``\forall a\in\dot{\mathcal{G}}(a\sse p\ra\dot{f}(a)=0)$''.
Suppose $p\in\mathcal{G}$.
Then in $V[\mathcal{G}]$,
for all $b\in\mathcal{G}$,  $b\cap p\in\mathcal{G}$,
so $f(b\cap p)=0$;
hence $f(b)=0$,
 since $f$ is monotone.
Thus, $p$ forces that $\dot{f}$ is constantly $0$ on $\dot{\mathcal{G}}$.
Letting $x=p$, the first case is finished.

\it Case 2. \rm Suppose now  that there is a $p'\le p$ such that $p'\forces \dot{f}(p')\ne 0$.
Without loss of generality, we may assume $p$ has this property and $p\in\bP_s$.
For the rest of the proof, we restrict to using only standard conditions;
that is, all conditions mentioned are assumed to be in $\bP_s$.
For each  pair $r,q$ such that $r\le q\le p$ and $r$ decides $\dot{f}(q)$,
let $\beta(r,q)$ denote the ordinal in $\om_1$  such that 
$r\forces\beta(r,q)=\min(\dot{f}(q))$.
Define
\begin{equation}
\beta=\min\{\beta(r,q):r\le q\le p\mathrm{\ and\ }r\mathrm{\ decides\ }\dot{f}(q)\}.
\end{equation}

Fix a pair $p_{-1}\le x_{-1}\le p$ for which $p_{-1}$ decides $\dot{f}(x_{-1})$ and such that  $\beta(p_{-1},x_{-1})=\beta$.

\begin{claim}
For all $v\le p_{-1}$ such that $v\sse x_{-1}$, $v\forces\beta=\min(\dot{f}(v))$.
\end{claim}

\begin{proof}
Let $v\le p_{-1}$ such that $v\sse x_{-1}$.  
For each $v'\le v$, there is a $v''\le v'$ deciding $\dot{f}(v)$.
Since $v''\le v \le p$,
we have that $\beta(v'',v)\ge \beta$.
On the other hand, $v\sse x_{-1}$ and $p\forces``\dot{f}$ is monotone'' imply that $v''\forces\dot{f}(v)\contains\dot{f}(x_{-1})$.
Hence, $v''\forces\beta\in\dot{f}(v)$.
By minimality of $\beta(v'',v)$ in $\dot{f}(v)$, 
it must be the case that $\beta(v'',v)=\beta$.
By density, we have that  $v\forces\min\dot{f}(v)=\beta$.
\end{proof}

We now build a decreasing sequence $p\ge x_{-1}\ge p_{-1}\ge x_0\ge p_0\ge x_1\ge p_1\ge\dots$ of members of $\bP_s$ such that 
each $p_n$ decides everything we need to know about $x_n$.
Diagonalizing the $x_n$, we will form an $x\sse p$ in $\bP_s$ such that $x$ forces the range of $\dot{f}$ to be countable.
The construction follows the same general outline as  the one given in the proof of Theorem \ref{thm.canon}.

Let $k\in\om$ be given, and suppose   we have chosen $x_j,p_j$ for all $-1\le j<k$.
If $k=0$, let $i_0=\min(\pi_1[x_{-1}])$;
if $k>0$, let $i_k=\min(\pi_1[x_{k-1}]\setminus (i_{k-1}+1))$.
Define $S_k$ to be 
 the set of all pairs $(s,t)$ such that 
$s\in\mathcal{P}(\{i_0,\dots,i_k\}\times(i_k+1))\cap x_{-1}$,
$t\in\mathcal{P}(\{i_0,\dots,i_k\})$,
and $\pi_1[s]\sse t$.
Fix an enumeration of $S_k$ as $(s_k^l,t_k^l)$, $l\le l_k:=|S_k|-1$.
Define $x_k^{-1}=x_{k-1}$ and $p_k^{-1}=p_{k-1}$.

Suppose $l\le  l_{k}$ and we have chosen $x_{k}^{l-1}$, $p^{l-1}_{k}$, and $\al_k^{l-1}$.
We choose $p_k^l\le x_k^l$ in $\bP_s$ and $\al_k^l\in\om_1$ as follows.

{\bf  If}
there are $v,q\in\bP_s$ with $q\le v\le  p_{k}^{l-1}$ and $\al<\om_1$ such that 
\begin{enumerate}
\item[(i)]
$q$ decides $\dot{f}(v)$;
\item[(ii)]
$v\cap(i_k+1)^2=s^{l}_k$;
\item[(iii)]
$\pi_1[v]\cap(i_k+1)=t_k^l$;
\item[(iv)]
For all $i\in t_{k}^{l}$, $v(i)\setminus (i_{k}+1)\sse  x_{k}^{l-1}(i)$;
\item[(v)]
$v\cap ((i_{k},\om)\times\om)\sse x_k^{l-1}$;
\item[(vi)]
$q\forces  \al\in\dot{f}(v)\setminus(\{\beta\}\cup \{\al_m^j:m<k$, $j\le l_m$,  $s_m^{j}=s_k^l\cap(i_{m}+1)^2$, and
$t_m^{j} =t_k^l\cap(i_{m}+1)\}$);
\end{enumerate}
{\bf then} 
take 
 $\al^l_k$ to be the minimum of all $\al$ satisfying (vi) for some pair $q\le v\le p_k^{l-1}$ satisfying (i) - (vi),
and take  some pair
$q_k^l\le v_k^l\le p_{k}^{l-1}$
satisfying (i) - (vi) with $\al^l_k$.
In this case, let 
\begin{align}
x_{k}^{l}=
& \bigcup \{\{i\}\times (v_{k}^l(i)\setminus (i_k+1)) : i\in t_{k}^l\}\cr
 &\cup\bigcup\{\{i\}\times x_k^{l-1}(i)   : i\in\{i_0,\dots,i_{k}\}\setminus t_{k}^l\}\cr
&\cup
(v_{k}^l\cap ((i_{k},\om)\times\om)).\cr
\end{align}
{\bf Otherwise}, let $p^{l}_{k}=p^{l-1}_{k}$, $x_{k}^{l}=x^{l-1}_{k}$, and $\al^l_k=\beta$; and
define
\begin{equation}
v_{k}^{l}= s_k^l\cup
\bigcup \{ \{i\}\times x_{k}^{l}(i): i\in t_{k}^{l}\}
\cup(x_{k}^{l}\cap((i_{k+1},\om)\times\om)).
\end{equation}

After the $l_k$ steps, let $p_k=p_k^{l_k}$ and $x_{k}=x_{k}^{l_k}$.
The construction guarantees that  for each $k\in\om$, $p_k\le x_k\le p_{k-1}$, $x_k\sse x_{k-1}$, and moreover  for each $l\le l_k$, $v_k^l\sse x_{-1}$.

\begin{claim}\label{fact.betastop}
Suppose $\al^l_k=\beta$ and   $v\sse v_k^l$ such that  $v\le p_k^l$.
For each $m\le k$, let
$j_m\le l_m$  be the integer satisfying
$s_m^{j_m}=v\cap(i_m+1)^2$ and
$t_m^{j_m}=\pi_1[v]\cap(i_m+1)$.
Then 
$v\forces\dot{f}(v)=\{\al^{j_m}_m:m\le k\}$.
Moreover, $\al_n^{j}=\beta$ for each $n>k$ and $j\le l_n$ such that 
$s^{j}_n\cap (i_k+1)^2=s_k^l$ and $t^{j}_n\cap (i_k+1)=t_k^l$.
\end{claim}

\begin{proof}
Suppose $\al^l_k=\beta$,
and suppose
 $v\le p_k^l$, $v\sse v_k^l$  and  satisfies (ii)  and (iii)   for $k,l$.
Since $v\sse v_k^l$,
automatically $v$ also satisfies (iv) and (v) for $k,l$.
Then for  each $q\le v$ satisfying (i) and each $\al<\om_1$,
(vi) does not hold, meaning that 
$q\forces \dot{f}(v)\sse \{\al^{j_m}_m:m\le k\}$.
Since it is dense below $v$ to have a $q$ for which (i)  holds, 
it follows that $v\forces\dot{f}(v)\sse \{\al^{j_m}_m:m\le k\}$.

On the other hand, for each $m\le k$,
since 
$v\sse v_m^{j_m}$, $p_m^{j_m}\forces \al_m^{j_m}\in\dot{f}(v_k^{j_k})$, and
 $p\forces \dot{f}$ is monotone,
it follows that $p_m^{j_m}\forces \al_m^{j_m}\in\dot{f}(v)$.
Hence, $p_k^l\forces \{\al^{j_m}_m:m\le k\}\sse\dot{f}(v)$.
Since $v\le p_k^l$, $v$ also forces $\{\al^{j_m}_m:m\le k\}\sse\dot{f}(v)$.

The second half follows since  $\al^l_k=\beta$ implies  (vi) fails for the pair $k,l$.
\end{proof}

Next,
diagonalize the $x_k$ similarly as in the proof of Theorem \ref{thm.canon}.
Let $k_0=0$.
Given $k_n$,
for each $j\le k_n$, choose an integer $a_{j,n}\in
x_{k_n}(i_j)\setminus(i_{k_n}+1)$.
Then take $k_{n+1}$ so that $i_{k_n+1}>\max\{a_{j,n}:j\le k_n\}$.
Let $x=\{(i_j,a_{j,n}):n<\om,\ j\le k_n\}$.
Note that $x\sse x_{-1}$, $x\le p_{-1}$, and  $x\sse^* x_k$ for all $k<\om$. Moreover, for each $n<\om$, we have $x\setminus(i_{k_n}+1)^2\sse x_{k_n}$.

\begin{claim}\label{claim.k_n}
For each $y\sse x$
 and each $n$,
if  $l\le l_{k_n}$
satisfies
$s_{k_n}^l = y\cap(i_{k_n}+1)^2$ and 
$t_{k_n}^l=\pi_1[y]\cap(i_{k_n}+1)$,
then
$x\forces\al_{k_n}^l\in\dot{f}(y)$.
\end{claim}

\begin{proof}
Let $y\sse x$, and suppose $n$ and $l$ satisfy the hypothesis.
Recall that $x\forces \al_{k_n}^l\in\dot{f}(v_{k_n}^l)$.
Since $x$ forces $\dot{f}$ to be monotone, it follows that for
each $v\sse v_{k_n}^l$, we have 
$x\forces \al_{k_n}^l\in\dot{f}(v)$.
  $y\sse x$ implies that $y\setminus (i_{k_n}+1)^2\sse x_{k_n}$.
This along with the fact that $y\cap (i_{k_n}+1)^2= v^l_{k_n}\cap (i_{k_n}+1)^2$
and $\pi_1[y]\cap (i_{k_n}+1)=v^l_{k_n}\cap (i_{k_n}+1)$ implies $y\sse v_{k_n}^l$.
It follows that 
$x\forces \al_{k_n}^l\in\dot{f}(y)$.
\end{proof}

\begin{claim}\label{claim.finite}
Let $n<\om$ and $l\le l_n$.
If  $\al_{k_n}^{l}\ne\beta$,
then for all $m< n$,
 $\al_{k_{n}}^{l}\ne\al_{k_m}^{j_m}$,
where $j_m\le l_m$ is the integer satisfying 
$s_{k_m}^{j_m}= s_{k_n}^{l}\cap(i_{k_m}+1)^2$ and $t_{k_m}^{j_m}=t_{k_n}^l\cap(i_{k_m}+1)$.
\end{claim}

\begin{proof}
Suppose  $\al_{k_n}^{l}\ne\beta$.
In the construction of $v_{k_n}^{l}$,
 (vi)
 implies that $\al_{k_n}^{l}\not\in\{\beta\}\cup\{\al_k^j:
k<k_n$, $j\le l_k$, $s^j_k=s_{k_n}^{l}\cap(i_{k}+1)^2$, and $t_k^j=t_{k_n}^{l}\cap (i_k+1)\}$.
In particular, 
$\al_{k_n}^{l}\not\in\{\al_{k_m}^{j_m}:m<n\}$.
\end{proof}

Now suppose $y\sse x$ and $y$ is a standard condition.
 For each $n<\om$, let $j_n\le l_{k_n}$ be the integer such that
$y\cap(i_{k_n}+1)^2=s_{k_n}^{j_n}$ and $\pi_1[y]\cap(i_{k_n}+1)=t_{k_n}^{j_n}$. 
Suppose that for all $n$, $\al_{k_n}^{j_n}\ne\beta$.
Then
by  Claims \ref{claim.k_n} and  \ref{claim.finite}, 
$x$ forces that $\dot{f}(y)$ is infinite.
This contradicts that $p$ forces $\dot{f}$ to have range in $[\om_1]^{<\om}$.
Thus, there is an $n<\om$ for which  $\al_{k_n}^{j_n}=\beta$.
It follows from Claim \ref{fact.betastop} that 
$y\forces \dot{f}(y)=\{\al_m^{j_m}:m\le k_n\}$,
where $j_m\le l_m$ is the integer such that $s_m^{j_m}=y\cap(i_m+1)^2$ and $t_m^{j_m}=\pi_1[y]\cap(i_m+1)$.

Thus, $x$ forces that the range of $\dot{f}$
is countable:
For each standard $y\sse x$, $y$ forces $\dot{f}(y)$ to be a finite subset of $\{\al_m^j:m,j<\om\}$.
If $z\sse x$ is a nonstandard condition in $\bP$, 
then for each $y\in\bP_s$ such that $y\sse z$, we have $y\forces \dot{f}(y)\contains\dot{f}(z)$;
so $z$ forces $\dot{f}(z)$ is  a finite subset of $\{\al_m^j:m,j<\om\}$.
Hence, 
if $x\in\mathcal{G}$, then in $V[\mathcal{G}]$,
 every $z\in\mathcal{G}$  has the property that  $f(z)$ is a finite subset of $\{\al_m^j:m,j<\om\}$.
Therefore, $x$ forces that $\dot{f}$ does not map  $\dot{\mathcal{G}}$ cofinally into $[\om_1]^{<\om}$.
\end{proof}

\section{Tukey Maps on Generic Ultrafilters} \label{sec:nicemaps}
The results in Section \ref{sec:nicemaps} and in Section \ref{sec:notbasicallygenerated} are due to Raghavan.
These results were obtained in mid-October 2012 during the Fields
Institute's thematic program on Forcing and its applications. 

Let $\XX \subseteq \Pset(\omega)$.
Recall that a map $\phi:\XX \rightarrow \Pset(\omega)$ is said to be
\emph{monotone} if $\forall a, b \in \XX \left[b \subseteq a \implies
  \phi(b) \subseteq \phi(a)\right]$. 
Such a map is said to be \emph{non-zero} if $\forall a \in
\XX\left[\phi(a) \neq 0 \right]$. 

We will show in this section that any monotone maps defined on the generic ultrafilter $\dot{\GG}$ have a ``nice'' canonical form similar to what is obtained in Section 4 of \cite{tukey}. 
This will imply that if $\GG$ is $(\V, \bbb{P})$-generic, then in
$\V\left[\GG\right]$ there are only $\c$ many ultrafilters that are
Tukey below $\GG$. 
This gives yet another proof that the generic ultrafilter is not of
the maximal cofinal type for directed sets of size continuum. 
The proof will go through the corresponding result for the sum of
selective ultrafilters indexed by a selective ultrafilter. 
Recall the following definitions and results which appear in \cite{tukey}.
\begin{df}\label{def:psi} 
  Let $\XX \subseteq \Pset(\omega)$ and let $\phi: \XX \rightarrow
  \Pset(\omega)$. Define ${\psi}_{\phi}: \Pset(\omega) \rightarrow
  \Pset(\omega)$ by ${\psi}_{\phi}(a) = \{k \in \omega: \forall b \in
  \XX\left[a \subseteq b \implies k \in \phi(b)\right] \} =
  \bigcap\{\phi(b): b \in \XX \wedge a \subseteq b\}$, for each $a \in
  \Pset(\omega)$.
\end{df}
\begin{la} [Lemma 16 of \cite{tukey}] \label{lem:findetermines} Let
  $\U$ be basically generated by $\BB \subseteq \U$.  Suppose moreover
  that $\forall {b}_{0}, {b}_{1} \in \BB \left[\; {b}_{0} \cap {b}_{1}
    \in \BB\right]$.  Let $\phi: \BB \rightarrow \Pset(\omega)$ be a
  monotone map such that $\phi(b) \neq 0$ for every $b \in \BB$.  Let
  $\psi = {\psi}_{\phi}$.  Then for every $b \in \BB$, ${\bigcup}_{s
    \in {\left[b\right]}^{< \omega}}{\psi(s)} \neq 0$.
\end{la}
Clearly Definition \ref{def:psi} and Lemma \ref{lem:findetermines}
apply to $\Pset({\omega}^{2})$ as well with the obvious modifications.

Let $\E$ and $\langle {\VV}_{n}: n \in \omega \rangle$ be selective
ultrafilters.  Put $\VV = \scr E\sm_n\scr V_n$.  Consider ${\BB}_{\VV}
= \{b \subseteq {\omega}^{2}: {\pi}_{1}(b) \in \E \ \text{and} \
\forall n \in {\pi}_{1}(b)\left[b(n) \in {\VV}_{n}\right]\}$.  Then
the following is easy to prove.  For a more general statement see
\cite{tukey}.
\begin{la} \label{lem:bv}
$\VV$ is basically generated by ${\BB}_{\VV}$. 
Moreover 
\begin{align*}
  \forall {b}_{0}, {b}_{1} \in {\BB}_{\VV} \left[\; {b}_{0} \cap
    {b}_{1} \in {\BB}_{\VV} \right].
\end{align*}
\end{la}
Thus Lemma \ref{lem:findetermines} can be applied to any sum over a
selective ultrafilter of selective ultrafilters.  We will do this
below to some selective ultrafilters that are generically added to a
ground model.
\begin{thm} \label{thm:nicemaps} Let $\GG$ be $(\V, \bbb{P})$-generic.
  In $\V\left[\GG\right]$, let $\phi: \GG \rightarrow \Pset(\omega)$
  be a monotone non-zero map.  Then there exist $P \subseteq
  {\left[{\omega}^{2} \right]}^{< \omega}$ and $\psi: P \rightarrow
  \omega$ such that
\begin{enumerate} 
 \item
  $\forall a \in \GG\left[P \cap {\left[a\right]}^{< \omega} \neq 0 \right]$.
\item $\forall a \in \GG \exists b \in \GG \cap
  {\left[a\right]}^{\omega}\forall s \in P \cap {\left[b\right]}^{<
    \omega}\left[\psi(s) \in \phi(b)\right]$.
\end{enumerate}
\end{thm}
\begin{proof}
Suppose that the theorem fails.
Fix $\dot{\phi} \in {\V}^{\bbb{P}}$ such that 
\begin{align*}
 \forces {\dot{\phi}: \dot{\GG} \rightarrow \Pset(\omega) \ \text{is a
     monotone non-zero map}}. 
\end{align*}
Fix a standard ${p}_{0} \in \bbb{P}$ such that for any $P \subseteq
{\left[ {\omega}^{2} \right]}^{< \omega}$ and $\psi: P \rightarrow
\omega$,
\begin{align*}
  {p}_{0} \; \forces \; ``\text{either} & \ \exists a \in
  \dot{\GG}\left[P \cap {\left[a\right]}^{< \omega} = 0 \right] \\ &
  \text{or} \ \exists a \in \dot{\GG} \forall b \in \dot{\GG} \cap
  {\left[a\right]}^{\omega} \exists s \in P \cap {\left[b\right]}^{<
    \omega}\left[\psi(s) \notin \dot{\phi}(b)\right]''.
\end{align*}
Let $\{\langle {p}_{\alpha}, {A}_{\alpha}, {\psi}_{\alpha} \rangle:
\alpha < {\c}^{\V} \}$ enumerate all triples $\langle p, A, \psi
\rangle$ such that $p \in \bbb{P}$ and $p \leq {p}_{0}$, $A \subseteq
{\left[ {\omega}^{2} \right]}^{< \omega}$, and $\psi: A \rightarrow
\omega$.  Define $\chi: \bbb{P} \rightarrow \Pset(\omega)$ by $\chi(p)
= \left\{k \in \omega: \exists q \leq p \left[q \forces k \in
    \dot{\phi}(p)\right]\right\}$.  Observe that if $q \leq p$, then
$q \forces p \in \dot{\GG}$, and hence $q \forces \dot{\phi}(p) \
\text{is defined}$.  Next, it is easy to check that $\chi$ is
monotone.  Moreover, $p \forces \dot{\phi}(p) \neq 0$.  Therefore, for
some $q \leq p$ and $k \in \omega$, $q \forces k \in \dot{\phi}(p)$,
whence $k \in \chi(p)$.  Thus $\chi$ is monotone and non-zero.  Now
build a sequence $\langle {q}_{\alpha}: \alpha < {\c}^{\V} \rangle$
with the following properties:
\begin{enumerate}
\item[(3)] ${q}_{\alpha} \in \bbb{P}$, ${q}_{\alpha}$ is standard, and
  ${q}_{\alpha} \subseteq {p}_{\alpha}$.
\item[(4)] either ${A}_{\alpha} \cap {\left[{q}_{\alpha}\right]}^{<
    \omega} = 0$ or for some $s \in {A}_{\alpha} \cap
  {\left[{q}_{\alpha}\right]}^{< \omega}$, ${\psi}_{\alpha}(s) \notin
  \chi({q}_{\alpha})$.
\end{enumerate} 
To see how to build such a sequence, fix $\alpha < {\c}^{\V}$.  Let
$\GG$ be $(\V, \bbb{P})$-generic with ${p}_{\alpha} \in \GG$.  Since
${p}_{\alpha} \leq {p}_{0}$, in $\V\left[\GG\right]$ either there is
$a \in \GG$ such that ${A}_{\alpha} \cap {\left[a\right]}^{< \omega} =
0$ or there is $a \in \GG$ such that for all $b \in \GG \cap
{\left[a\right]}^{\omega}$, there exists $s \in {A}_{\alpha} \cap
{\left[b\right]}^{< \omega}$ such that ${\psi}_{\alpha}(s) \notin
\dot{\phi}\left[\GG\right](b)$.  Suppose that the first case happens.
Let ${q}_{\alpha}$ be a standard element of $\bbb{P}$ such that
${q}_{\alpha} \subseteq {p}_{\alpha} \cap a$.  Then
${\left[{q}_{\alpha}\right]}^{< \omega} \cap {A}_{\alpha} \subseteq
{\left[a\right]}^{< \omega} \cap {A}_{\alpha} = 0$.

Now suppose that the second case happens in $\V\left[\GG\right]$.
Working in $\V\left[\GG\right]$ fix $a \in \GG$ as in the second case.
Let $b \in \GG$ be standard such that $b \subseteq {p}_{\alpha} \cap
a$.  Since $b \in \GG \cap {\left[a\right]}^{\omega}$ there is $s \in
{A}_{\alpha} \cap {\left[b\right]}^{< \omega}$ such that
${\psi}_{\alpha}(s) \notin \dot{\phi}\left[\GG\right](b)$.  Find
${q}^{\ast} \in \GG$ such that (in $\V$) ${q}^{\ast} \forces
{\psi}_{\alpha}(s) \notin \dot{\phi}(b)$.  Let $q \in \GG$ be standard
so that $q \subseteq b \cap {q}^{\ast}$.  Back in $\V$, define
${q}_{\alpha}$ as follows.  For $n \in {\pi}_{1}(s)$, put
${q}_{\alpha}(n) = b(n)$.  If $n \in \omega - {\pi}_{1}(s)$, then
${q}_{\alpha}(n) = q(n)$.  Note that ${q}_{\alpha} \in \bbb{P}$, it is
standard, and $s \subseteq {q}_{\alpha} \subseteq b \subseteq
{p}_{\alpha}$.  Moreover, if $\langle n, m \rangle \in {q}_{\alpha} -
q$, then $n \in {\pi}_{1}(s)$.  As ${\pi}_{1}(s)$ is finite,
${q}_{\alpha} - q \in {\left( {\F}^{\otimes2}\right)}^{\ast}$.
Therefore, ${q}_{\alpha} \leq q$ and ${q}_{\alpha} \forces
{\psi}_{\alpha}(s) \notin \dot{\phi}(b)$.  Note that $s \in
{A}_{\alpha} \cap {\left[{q}_{\alpha}\right]}^{< \omega}$.  To see
that ${\psi}_{\alpha}(s) \notin \chi({q}_{\alpha})$, suppose for a
contradiction that there is $r \leq {q}_{\alpha}$ such that $r \forces
{\psi}_{\alpha}(s) \in \dot{\phi}({q}_{\alpha})$.  As ${q}_{\alpha}
\subseteq b$, $r \forces {\psi}_{\alpha}(s) \in \dot{\phi}(b)$, which
is impossible.  This completes the construction of ${q}_{\alpha}$.

We wish to apply Lemma \ref{lem:findetermines} to a sum of selective
ultrafilters indexed by another selective ultrafilters.  These
selective ultrafilters are obtained generically as follows.  Let $\E$
be $(\V, \Pset(\omega) / \F)$-generic with ${\pi}_{1}({p}_{0}) \in
\E$.  In $\V\left[\E\right]$ consider the poset $\bbb{Q} = {\left(
    \Pset(\omega) / \F \right)}^{\omega}$.  Define ${x}_{0} \in
\bbb{Q}$ as follows.  For any $n \in {\pi}_{1}({p}_{0})$, ${x}_{0}(n)
= {p}_{0}(n)$.  For any $n \notin {\pi}_{1}({p}_{0})$, ${x}_{0}(n) =
\omega$.  Let $G$ be $(\V\left[\E\right], \bbb{Q})$-generic with
${x}_{0} \in G$.  In $\V[\E][G]$ define for each $n \in \omega$,
${\VV}_{n} = \{x(n): x \in G\}$.  It is clear that $\E$ and the
${\VV}_{n}$'s are selective ultrafilters in
$\V\left[\E\right]\left[G\right]$.  Put $\VV = \scr E\sm_n\scr V_n$.
Then by Lemma \ref{lem:bv} $\VV$ is basically generated by
${\BB}_{\VV}$ and $\forall {b}_{0}, {b}_{1} \in {\BB}_{\VV} \left[\;
  {b}_{0} \cap {b}_{1} \in {\BB}_{\VV} \right]$.  Note that
${\BB}_{\VV} \subseteq \VV \subseteq \bbb{P}$.  Put $\phi = \chi
\restrict {\BB}_{\VV}$.  The hypotheses of Lemma
\ref{lem:findetermines} are satisfied.  Put $A = \{s \in
{\left[{\omega}^{2}\right]}^{< \omega}: {\psi}_{\phi}(s) \neq 0 \}$.
Define $\psi: A \rightarrow \omega$ by $\psi(s) =
\min({\psi}_{\phi}(s))$ for any $s \in A$.  We claim that there exists
$\alpha < {\c}^{\V}$ such that ${q}_{\alpha} \in {\BB}_{\VV}$ and
${A}_{\alpha} = A$ and ${\psi}_{\alpha} = \psi$.  Suppose for a moment
that this claim is true.  Applying Lemma \ref{lem:findetermines} to
${q}_{\alpha}$ find $s \in {\left[{q}_{\alpha}\right]}^{< \omega}$
such that ${\psi}_{\phi}(s) \neq 0$.  So $s \in {A}_{\alpha} \cap
{\left[{q}_{\alpha}\right]}^{< \omega}$.  Moreover, by the definition
of ${\psi}_{\phi}$, for any $t \in {A}_{\alpha} \cap
{\left[{q}_{\alpha}\right]}^{< \omega}$, ${\psi}_{\phi}(t) \subseteq
\phi({q}_{\alpha}) = \chi({q}_{\alpha})$.  This means that for every
$t \in {A}_{\alpha} \cap {\left[{q}_{\alpha}\right]}^{< \omega}$,
${\psi}_{\alpha}(t) \in \chi({q}_{\alpha})$.  But this contradicts the
way ${q}_{\alpha}$ was constructed.

To prove the claim first note that $A$ and $\psi$ are in $\V$.  In
$\V\left[\E\right]$ define $D(A, \psi)$ to be the collection of all $y
\in \bbb{Q}$ such that
\begin{align*}
  \exists a \in \E \exists \alpha < {\c}^{\V}
  \left[{\pi}_{1}({q}_{\alpha}) = a, y \restrict a = {q}_{\alpha},
    {A}_{\alpha} = A, \ \text{and} \ {\psi}_{\alpha} = \psi \right].
\end{align*}  
We argue that $D(A, \psi)$ is dense below ${x}_{0}$.  Fix $x \in
\bbb{Q}$ with $x \leq {x}_{0}$.  Note that $x \in \V$.  Working in
$\V$ define $D(x, A, \psi) = \{{\pi}_{1}({q}_{\alpha}): \alpha <
{c}^{\V}, {q}_{\alpha} \subseteq x \restrict \omega, {A}_{\alpha} = A,
\ \text{and} \ {\psi}_{\alpha} = \psi \}$.  To see that $D(x, A,
\psi)$ is dense below ${\pi}_{1}({p}_{0})$ fix $a \in
{\left[{\pi}_{1}({p}_{0})\right]}^{\omega}$.  Put $p = x \restrict a$
and note that $p \in \bbb{P}$ and that $p \leq {p}_{0}$.  Therefore,
there exists $\alpha < {\c}^{\V}$ such that ${p}_{\alpha} = p$,
${A}_{\alpha} = A$, and ${\psi}_{\alpha} = \psi$.  Thus ${q}_{\alpha}
\subseteq x \restrict a \subseteq x \restrict \omega$.  Also
${\pi}_{1}({q}_{\alpha}) \subseteq a$.  Therefore
${\pi}_{1}({q}_{\alpha})$ is as needed.  Back in $\V\left[\E\right]$,
fix $a \in \E$ and $\alpha < {\c}^{\V}$ such that
${\pi}_{1}({q}_{\alpha}) = a$, ${q}_{\alpha} \subseteq x \restrict
\omega$, ${A}_{\alpha} = A$, and ${\psi}_{\alpha} = \psi$.  For $n \in
a$, put $y(n) = {q}_{\alpha}(n)$.  For $n \in \omega - a$, put $y(n) =
x(n)$.  Then $y \in \bbb{Q}$ and $y \leq x$.  It is clear that $y \in
\bbb{Q}$ and that $y \leq x$.  Also $y \restrict a = {q}_{\alpha}$ and
so it is clear that $y$ is as needed.  So in
$\V\left[\E\right]\left[G\right]$, there is $y \in G$, $a \in \E$, and
$\alpha < {\c}^{\V}$ such that ${\pi}_{1}({q}_{\alpha}) = a$, $y
\restrict a = {q}_{\alpha}$, ${A}_{\alpha} = A$, and ${\psi}_{\alpha}
= \psi$.  Since ${\pi}_{1}({q}_{\alpha}) = a \in \E$ and for all $n
\in {\pi}_{1}({q}_{\alpha})$, ${q}_{\alpha}(n) = y(n) \in {\VV}_{n}$,
${q}_{\alpha} \in {\BB}_{\VV}$, and we are done.
\end{proof}
Now we show that the conclusion of Theorem 17 of \cite{tukey}, which
was proved there to hold for all ultrafilters that are basically
generated by a base that is closed under finite intersections, also
holds for the generic ultrafilter $\dot{\GG}$.
\begin{df} \label{def:UP} Let $\U$ be an ultrafilter on
  ${\omega}^{2}$, and let $P \subseteq {\left[{\omega}^{2}\right]}^{<
    \omega} - \{0\}$.  We define ${\U}(P) = \{A \subseteq P: \exists a
  \in \U\left[P \cap {\left[a\right]}^{< \omega} \subseteq A
  \right]\}$.
\end{df}   
If $\forall a \in \U\left[\lc P \cap {\left[a \right]}^{< \omega} \rc
  = \omega\right]$, then ${\U}(P)$ is a proper, non-principal filter
on $P$.  The following theorem says that any Tukey reduction from
$\dot{\GG}$ is given by a Rudin-Keisler reduction from $\dot{\GG}(P)$
for some $P$.
\begin{thm} \label{thm:canonical} Let $\GG$ be $(\V,
  \bbb{P})$-generic.  In $\V\left[\GG\right]$, let $\VV$ be an
  arbitrary ultrafilter so that $\VV \; {\leq}_{T} \; \GG$.  Then
  there is $P \subseteq {\left[{\omega}^{2}\right]}^{< \omega} -
  \{0\}$ such that
  \begin{enumerate}
    \item
    $\forall t, s \in P \left[t \subseteq s \implies t = s \right]$		
    \item
    $\GG(P) \; {\equiv}_{T} \; \GG$
    \item
    $\VV \; {\leq}_{RK} \; \GG(P)$
  \end{enumerate}
\end{thm}  
\begin{proof}
  The proof is almost the same as the proof of Theorem 17 of
  \cite{tukey}.  Work in $\V\left[\GG\right]$.  Fix an ultrafilter
  $\VV$ and a map $\phi: \GG \rightarrow \VV$ which is monotone and
  cofinal in $\VV$.  Since $\phi$ is monotone and non-zero, fix $A
  \subseteq {\left[{\omega}^{2}\right]}^{< \omega}$ and $\psi: A
  \rightarrow \omega$ as in Theorem \ref{thm:nicemaps}.  First we
  claim that $0 \notin A$.  Indeed suppose for a contradiction that $0 \in A$ and let $k = \psi(0)$.  Let $e \in \VV$ be such that $k
  \notin e$ and let $a \in \GG$ be such that $\phi(a) \subseteq e$.
  By (2) of Theorem \ref{thm:nicemaps} there is $b \in \GG \cap
  {\left[a\right]}^{\omega}$ such that for all $s \in A \cap
  {\left[b\right]}^{< \omega}$, $\psi(s) \in \phi(b)$.  However, $0
  \in A \cap {\left[b\right]}^{< \omega}$, and so $k = \psi(0) \in
  \phi(b) \subseteq \phi(a) \subseteq e$, a contradiction.  Thus $0
  \notin A$.  Define
\begin{align*}
  P = \{s \in A: s \ \text{is minimal in} \ A \ \text{with respect to}
  \ \subseteq \}.
\end{align*}
It is clear that $P \subseteq {\left[{\omega}^{2}\right]}^{< \omega} -
\{0\}$ and that $P$ satisfies (1) by definition.

Next, for any $a \in \GG$, ${\bigcup}{\left(P \cap {\left[a\right]}^{< \omega}\right)} \in \GG$.
To see this, fix $a \in \GG$, and suppose that $a - \left({\bigcup}{\left(P \cap {\left[a\right]}^{< \omega}\right)}\right) \in \GG$.  By (1) of Theorem \ref{thm:nicemaps}, fix $s \in A$ with $s \subseteq a - \left({\bigcup}{\left(P \cap {\left[a\right]}^{< \omega}\right)}\right)$.
However there is $t \in P$ with $t \subseteq s$, whence $t = 0$, an impossibility.
It follows from this that for each $a \in \GG$, $P \cap {\left[a\right]}^{< \omega}$ is infinite.

Next, verify that $\GG(P) \; {\equiv}_{T} \; \GG$.
Define $\chi: \GG \rightarrow \GG(P)$ by $\chi(a) = P \cap {\left[a\right]}^{< \omega}$, for each $a \in \GG$.  This map is clearly monotone and cofinal in $\GG(P)$.  So $\chi$ is a convergent map.  On the other hand, $\chi$ is also Tukey.
To see this, fix $\XX \subseteq \GG$, unbounded in $\GG$.
Assume that $\{\chi(a): a \in \XX\}$ is bounded in $\GG(P)$.
So there is $b \in \GG$ such that $P \cap {\left[b\right]}^{< \omega}
\subseteq P \cap {\left[a\right]}^{< \omega}$ for each $a \in \XX$.
However $c = {\bigcup}{\left(P \cap {\left[b\right]}^{< \omega}\right)} \in \GG$.
Now, it is clear that $c \subseteq a$, for each $a \in \XX$, a contradiction.

Next, check that $\VV \; {\leq}_{RK} \; \GG(P)$.  Define $f: P
\rightarrow \omega$ by $f = \psi \restrict P$.  Fix $e \subseteq
\omega$, and suppose first that ${f}^{-1}(e) \in \GG(P)$.  Fix $a \in
\GG$ with $P \cap {\left[a\right]}^{< \omega} \subseteq {f}^{-1}(e)$.
If $e \notin \VV$, then $\omega - e \in \VV$, and there exists $c \in
\GG$ with $\phi(c) \subseteq \omega - e$.  By (2) of Theorem
\ref{thm:nicemaps} fix $b \in \GG \cap {\left[a \cap
    c\right]}^{\omega}$ such that for all $s \in A \cap
{\left[b\right]}^{< \omega}$, $\psi(s) \in \phi(b)$.  By (1) of
Theorem \ref{thm:nicemaps}, fix $s \in A \cap {\left[b\right]}^{<
  \omega}$.  Fix $t \subseteq s$ with $t \in P$.  Let $k = f(t) =
\psi(t)$.  As $t \subseteq s \subseteq b \subseteq a$, $t \in P \cap
{\left[a\right]}^{< \omega} \subseteq {f}^{-1}(e)$.  Thus $k \in e$.
On the other hand, since $t \in A \cap {\left[b\right]}^{< \omega}$,
$\psi(t) \in \phi(b)$.  So $k \in \phi(b) \subseteq \phi(c) \subseteq
\omega - e$, a contradiction.

Next, suppose that $e \in \VV$.  By cofinality of $\phi$, there is $a
\in \GG$ such that $\phi(a) \subseteq e$.  Applying (2) of Theorem
\ref{thm:nicemaps}, fix $b \in \GG \cap {\left[a\right]}^{\omega}$
such that for all $s \in A \cap {\left[b\right]}^{< \omega}$, $\psi(s)
\in \phi(b)$.  Now, if $s \in P \cap {\left[b\right]}^{< \omega}$,
then $f(s) = \psi(s) \in \phi(b) \subseteq \phi(a) \subseteq e$.
Therefore, $P \cap {\left[b\right]}^{< \omega} \subseteq {f}^{-1}(e)$,
whence ${f}^{-1}(e) \in \GG(P)$.
\end{proof}
An immediate corollary of Theorem \ref{thm:canonical} is that if $\GG$
is $(\V, \bbb{P})$-generic, then in $\V\left[\GG\right]$, $\{\VV: \VV
\ \text{is an ultrafilter on} \ \omega \ \text{and} \ \VV \;
{\leq}_{T} \; \GG\}$ has size $\c$.  This is because there are only
$\c$ many sets $P \subseteq {\left[ {\omega}^{2} \right]}^{< \omega} -
\{0\}$, and for each such $P$ there are only $\c$ many ultrafilters
that are RK below $\GG(P)$.

\section{Generic Ultrafilters are not Basically
  Generated} \label{sec:notbasicallygenerated} 

In this section, we show that the generic ultrafilter is not basically
generated.  This is the first consistent example of an ultrafilter
that is not basically generated and whose cofinal type is not maximal.
Thus our result establishes the consistency of the statement
\begin{align*}
  ``\exists \U\left[\U \ \text{is not basically generated} \
    \text{and} \ {\left[{\omega}_{1}\right]}^{< \omega} \;
    {\not\leq}_{T} \; \U\right]''.
\end{align*}

Cardinal invariants of the Boolean algebra $\Pset({\omega}^{2}) /
{\F}^{\otimes2}$ were considered in \cite{tisomega1} and
\cite{hernandez}.  Recall that for a Boolean algebra $\BB$,
$\mathfrak{t}(\BB)$ is the least $\kappa$ such that there is a
sequence $\langle {b}_{\alpha}: \alpha < \kappa \rangle$ of elements
of $\BB \setminus \{0\}$ such that for all $\alpha < \beta <
\kappa\left[{b}_{\alpha} \geq {b}_{\beta}\right]$ and there does not
exist $b \in \BB \setminus \{0\}$ such that $\forall \alpha < \kappa
\left[b \leq {b}_{\alpha}\right]$.  $\mathfrak{h}(\BB)$ is the
distributivity number of $\BB$.  Szyma\'{n}ski and Zhou showed in
\cite{tisomega1} that $\mathfrak{t}\left(\Pset({\omega}^{2}) /
  {\F}^{\otimes2}\right)$ is provably equal to ${\omega}_{1}$ in
$\ZFC$.  Hern{\'a}ndez-Hern{\'a}ndez showed in \cite{hernandez} that
it is consistent to have $\mathfrak{h}\left(\Pset({\omega}^{2}) /
  {\F}^{\otimes2}\right) < \mathfrak{h}\left(\Pset(\omega) \slash
  \Fin\right)$.  Our proof of Theorem \ref{thm:notbasicallygenerated}
is partly inspired by these well-known facts, though our construction
is different.

For ease of notation, throughout this section, we use $\I$ to denote
the dual ideal to ${\F}^{\otimes2}$.  In other words, $\I = {\left(
    {\F}^{\otimes2} \right)}^{\ast}$.  Also, let $a \in \cube$.  If
$\A \subseteq \Pset(a)$, $\I(\A)$ denotes the ideal on $a$ generated
by $\A$ together with the Fr\'echet ideal on $a$.
\begin{thm} \label{thm:notbasicallygenerated}
$\forces {\dot{\GG} \ \text{is not basically generated}}$.
\end{thm}
\begin{proof}
Let $\dot{\BB} \in \VP$ be such that
\begin{enumerate}
  \item
  $\forces \dot{\BB} \subseteq \dot{\GG}$
  \item
  $\forces \forall a \in \dot{\GG} \exists b \in \dot{\BB}\left[b
    \subseteq a \right]$ 
\end{enumerate}
Let ${p}^{\ast} \in \bbb{P}$ be standard such that
\begin{align*}
  {p}^{\ast} \forces ``&\text{every convergent sequence from} \
  \dot{\BB}  \\ & \text{contains an infinite sub-sequence bounded in}
  \ \dot{\GG}'' 
\end{align*}
Now build two sequences $\{{p}_{\alpha}: \alpha < {\omega}_{1}\}$ and
$\{{x}_{\alpha}: \alpha < {\omega}_{1}\}$ with the following
properties.
\begin{enumerate}
\item[(3)] ${p}_{\alpha} \subseteq {p}^{\ast}$, both ${p}_{\alpha}$
  and ${x}_{\alpha}$ are elements of $\bbb{P}$, ${p}_{\alpha}$ is
  standard, ${p}_{\alpha} \subseteq {x}_{\alpha}$, and ${p}_{\alpha}
  \forces {x}_{\alpha} \in \dot{\BB}$.
\item[(4)] $\forall \xi < \alpha \left[{x}_{\alpha} \leq
    {p}_{\xi}\right]$ (therefore, $\forall \xi < \alpha
  \left[{p}_{\alpha} \leq {x}_{\alpha} \leq {p}_{\xi}\right]$).
\item[(5)] $\forall n \in \omega \left[\{\xi < \alpha: \lc \left(
      {x}_{\alpha} \cap {x}_{\xi} \right) (n) \rc = \omega \} \
    \text{is finite}\right]$.
\item[(6)] for each $\alpha < {\omega}_{1}$ and $n \in
  {\pi}_{1}({p}_{\alpha})$ let $F(\alpha, n) = \{\xi \leq \alpha:
  {p}_{\alpha}(n) \; {\subseteq}^{\ast} \; {x}_{\xi}(n)\}$.  Note that
  $\alpha \in F(\alpha, n)$.  Let $G(\alpha, n) = \{{p}_{\alpha}(n)
  \cap {x}_{\xi}(n): \xi \in {\omega}_{1} - F(\alpha, n)\}$.  Then
  $\I(G(\alpha, n))$ is a proper ideal on ${p}_{\alpha}(n)$ for each
  $\alpha < {\omega}_{1}$ and $n \in {\pi}_{1}({p}_{\alpha})$.
\end{enumerate}
Suppose for a moment that such sequences can be constructed.  Let
$\delta < {\omega}_{1}$ and $\{{\alpha}_{0} < {\alpha}_{1} < \dotsb \}
\subseteq \delta$ be such that $\{\langle {x}_{{\alpha}_{i}},
{p}_{{\alpha}_{i}} \rangle: i \in \omega\}$ converges to $\langle
{x}_{\delta}, {p}_{\delta} \rangle$ (with respect to the usual
topology on $\Pset(\omega) \times \Pset(\omega)$).  Note that for each
$i \in \omega$, ${p}_{\delta} \leq {p}_{{\alpha}_{i}}$.  Therefore
${p}_{\delta} \forces \{{x}_{{\alpha}_{i}}: i < \omega\} \cup
\{{x}_{\delta}\} \subseteq \dot{\BB}$.  Since ${p}_{\delta} \leq
{p}^{\ast}$ and since $\bbb{P}$ does not add any countable sets of
ordinals, there exist $X \in \cube$ and a standard $q \in \bbb{P}$
such that $q \subseteq {x}_{\delta}$ and $\forall i \in X \left[q
  \subseteq {x}_{{\alpha}_{i}}\right]$.  Fix $n \in {\pi}_{1}(q)$.
Then for each $i \in X$, $q(n) \subseteq {x}_{\delta}(n) \cap
{x}_{{\alpha}_{i}}(n)$.  So $\{{\alpha}_{i}: i \in X\} \subseteq
\{\alpha < \delta: \lc \left( {x}_{\delta} \cap {x}_{\alpha} \right)
(n) \rc = \omega \}$, contradicting (5).

To see how to build such sequences, note first that if $\delta \leq
{\omega}_{1}$ is a limit ordinal and if for each $\beta < \delta$ the
sequences $\langle {x}_{\alpha}: \alpha < \beta \rangle$ and $\langle
{p}_{\alpha}: \alpha < \beta \rangle$ do not contain any witnesses
violating clauses (3)-(6), then the sequences $\langle {x}_{\alpha}:
\alpha < \delta \rangle$ and $\langle {p}_{\alpha}: \alpha < \delta
\rangle$ do not contain any such witnesses either.  We may thus
concentrate on extending two such given sequences by one step.
Therefore, fix $\alpha < {\omega}_{1}$ and assume that $\langle
{x}_{\xi}: \xi < \alpha \rangle$ and $\langle {p}_{\xi}: \xi < \alpha
\rangle$ are given to us.  We only need to worry about finding
${x}_{\alpha}$ and ${p}_{\alpha}$.  First if $\alpha = 0$, then fix a
$(\V, \bbb{P})$-generic $\GG$ with ${p}^{\ast} \in \GG$.  In
$\V\left[\GG\right]$ fix ${x}_{0} \in \dot{\BB}\left[\GG\right]$ with
${x}_{0} \subseteq {p}^{\ast}$.  In $\V$, fix a standard ${p}_{0} \in
\bbb{P}$ such that ${p}_{0} \subseteq {x}_{0}$ and ${p}_{0} \forces
{x}_{0} \in \dot{\BB}$.  It is clear that (3) is satisfied, and
(4)-(6) are trivially true.  So assume $\alpha > 0$.  Let
$\{{\xi}_{n}: n \in \omega\}$ enumerate $\alpha$, possibly with
repetitions.  For each $n \in \omega$, let ${\zeta}_{n} =
\max\{{\xi}_{i}: i \leq n\}$.  Note that for each $i \leq n$,
${p}_{{\zeta}_{n}} \leq {p}_{{\xi}_{i}}$.  So it is possible to find a
sequence of elements of $\omega$, $\{{k}_{0} < {k}_{1} < \dotsb \}$, such that for each $n \in \omega$, ${k}_{n} \in
{\pi}_{1}({p}_{{\zeta}_{n}})$ and for each $i \leq n$,
${p}_{{\zeta}_{n}}({k}_{n}) \; {\subseteq}^{\ast} \;
{p}_{{\xi}_{i}}({k}_{n})$.  Define $p \subseteq \omega \times \omega$
as follows.  If $m \notin \{{k}_{0} < {k}_{1} < \dotsb\}$, then $p(m)
= 0$.  Suppose $m = {k}_{n}$.  Put $G({\zeta}_{n}, m, \alpha) =
\{{p}_{{\zeta}_{n}}(m) \cap {x}_{\xi}(m): \xi \in \alpha -
F({\zeta}_{n}, m)\}$.  By (6) $\I(G({\zeta}_{n}, m, \alpha))$ is a
proper ideal on ${p}_{{\zeta}_{n}}(m)$.  Since this ideal is countably
generated, it is possible to find $p(m) \in
{\left[{p}_{{\zeta}_{n}}(m)\right]}^{\omega}$ such that
\begin{enumerate}
\item[(7)] for all $a \in \I(G({\zeta}_{n}, m, \alpha))$, $\lc p(m)
  \cap a \rc < \omega$
\item[(8)] for all $a \in \I(G({\zeta}_{n}, m, \alpha))$, $\lc \left(
    \omega - a \right) \cap \left(\omega - p(m) \right)\rc = \omega$.
\end{enumerate}
Note that $p \in \bbb{P}$.  Furthermore, note that if $i \in \omega$,
then for any $n \geq i$, $p({k}_{n}) \subseteq
{p}_{{\zeta}_{n}}({k}_{n}) \; {\subseteq}^{\ast} \;
{p}_{{\xi}_{i}}({k}_{n})$.  Hence for all $\xi < \alpha$, $p \leq
{p}_{\xi}$.  Next, fix $m \in \omega$ and suppose that $(p \; \cap \;
{x}_{\xi})(m)$ is infinite for some $\xi < \alpha$.  Then $m =
{k}_{n}$ for some (unique) $n$ and $\xi \in F({\zeta}_{n}, m)$.
However $F({\zeta}_{n}, m)$ must be a finite set.  This is because if
$\xi \in F({\zeta}_{n}, m)$, then $\xi \leq {\zeta}_{n}$ and
${p}_{{\zeta}_{n}}(m) \; {\subseteq}^{\ast} \; {x}_{\xi}(m) \cap
{x}_{{\zeta}_{n}}(m)$, and so since $m \in
{\pi}_{1}({p}_{{\zeta}_{n}})$, $F({\zeta}_{n}, m) \subseteq
\{{\zeta}_{n}\} \cup \{\xi < {\zeta}_{n}: \lc \left( {x}_{{\zeta}_{n}}
  \cap {x}_{\xi} \right) (m) \rc = \omega\}$, which is a finite set.
So for any $m \in \omega$, $\{\xi < \alpha: \lc \left(p \cap
  {x}_{\xi}\right)(m)\rc = \omega\}$ is finite.  Finally, note that $p
\subseteq {p}^{\ast}$.

Let $\GG$ be $(\V, \bbb{P})$-generic with $p \in \GG$.  In
$\V\left[\GG\right]$, let ${x}_{\alpha} \in \dot{\BB}\left[\GG\right]$
with ${x}_{\alpha} \subseteq p$.  In $\V$, let ${p}_{\alpha} \in
\bbb{P}$ be standard such that ${p}_{\alpha} \subseteq {x}_{\alpha}
\subseteq p$ and ${p}_{\alpha} \forces {x}_{\alpha} \in \dot{\BB}$.
It is clear that (3)-(5) are satisfied by $\langle {x}_{\xi}: \xi \leq
\alpha \rangle$ and $\langle {p}_{\xi}: \xi \leq \alpha\rangle$.

We only need to check that (6) is satisfied.  There are several cases
to consider here.  First fix $m \in \omega$ and suppose that $m \in
{\pi}_{1}({p}_{\alpha})$.  As ${p}_{\alpha} \subseteq p$, $m =
{k}_{n}$ for some (unique) $n \in \omega$.  Now if $\xi < \alpha$ and
${p}_{\alpha}(m) \; {\subseteq}^{\ast} \; {x}_{\xi}(m)$, then $p(m)
\cap {x}_{\xi}(m)$ is infinite and so $\xi \in F({\zeta}_{n}, m)$.  On
the other hand if $\xi \in F({\zeta}_{n}, m)$, then ${p}_{\alpha}(m)
\subseteq p(m) \subseteq {p}_{{\zeta}_{n}}(m) \; {\subseteq}^{\ast} \;
{x}_{\xi}(m)$, whence $\xi \in F(\alpha, m)$.  Therefore, $F(\alpha,
m) = \{\alpha\} \cup F({\zeta}_{n}, m)$.  Put $G(\alpha, m, \alpha +
1) = \{{p}_{\alpha}(m) \cap {x}_{\xi}(m): \xi \in \left( \alpha + 1
\right) - F(\alpha, m)\}$.  By (7) it is clear that $\I(G(\alpha, m,
\alpha + 1))$ is the Frechet ideal on ${p}_{\alpha}(m)$.  This takes
care of $\alpha$.  Next, suppose $\xi < \alpha$ and $m \in
{\pi}_{1}({p}_{\xi})$.  Put $G(\xi, m, \alpha) = \{{p}_{\xi}(m) \cap
{x}_{\zeta}(m): \zeta \in \alpha - F(\xi, m)\}$ and put $G(\xi, m,
\alpha + 1) = \{{p}_{\xi}(m) \cap {x}_{\zeta}(m): \zeta \in \left(
  \alpha + 1\right) - F(\xi, m)\}$.  We know that $\I(G(\xi, m,
\alpha))$ is a proper ideal on ${p}_{\xi}(m)$ and it is clear that
$\I(G(\xi, m, \alpha)) = \I(G(\xi, m, \alpha + 1))$ unless
${p}_{\xi}(m) \cap {x}_{\alpha}(m) \notin \I(G(\xi, m, \alpha))$.
Suppose this is the case.  In particular, ${p}_{\xi}(m) \cap
{x}_{\alpha}(m)$ is infinite.  Since ${x}_{\alpha}(m) \subseteq p(m)$
and ${p}_{\xi}(m) \subseteq {x}_{\xi}(m)$, it follows that $m =
{k}_{n}$ for some (unique) $n$ and $\xi \in F({\zeta}_{n}, m)$.
Moreover, if $\xi < {\zeta}_{n}$, then since ${\zeta}_{n} \in \alpha -
F(\xi, m)$ and since ${x}_{\alpha}(m) \subseteq p(m) \subseteq
{p}_{{\zeta}_{n}}(m) \subseteq {x}_{{\zeta}_{n}}(m)$, we have that
${p}_{\xi}(m) \cap {x}_{\alpha}(m) \subseteq {p}_{\xi}(m) \cap
{x}_{{\zeta}_{n}}(m) \in \I(G(\xi, m, \alpha))$.  Therefore, $\xi =
{\zeta}_{n}$.  Thus we need to show that $\I(G({\zeta}_{n}, m, \alpha
+ 1))$ is a proper ideal on ${p}_{{\zeta}_{n}}(m)$.  For this it
suffices to show that $\omega - \left( {p}_{{\zeta}_{n}}(m) \cap
  {x}_{\alpha}(m)\right) \notin \I(G({\zeta}_{n}, m, \alpha))$.  Note
that since ${x}_{\alpha}(m) \subseteq p(m) \subseteq
{p}_{{\zeta}_{n}}(m)$, ${p}_{{\zeta}_{n}}(m) \cap {x}_{\alpha}(m) =
{x}_{\alpha}(m)$.  However it is clear from (8) that $\omega -
{x}_{\alpha}(m) \notin \I(G({\zeta}_{n}, m, \alpha))$ and we are done.
\end{proof}

\end{document}